\documentclass[12pt,reqno]{amsart}

\usepackage[letterpaper,margin=1in]{geometry}
\usepackage{amsmath,amssymb,mathtools,mathrsfs,esint}
\usepackage{enumitem}
\usepackage{microtype}
\usepackage[colorlinks=true,linkcolor=blue,citecolor=blue,urlcolor=blue]{hyperref}
\usepackage[nameinlink,capitalise,noabbrev]{cleveref}

\allowdisplaybreaks
\numberwithin{equation}{section}

\newtheorem{theorem}{Theorem}[section]
\newtheorem{proposition}[theorem]{Proposition}
\newtheorem{lemma}[theorem]{Lemma}
\newtheorem{corollary}[theorem]{Corollary}
\newtheorem*{unnumberedtheorem}{Theorem}
\newtheorem*{unnumberedlemma}{Lemma}
\newtheorem*{unnumberedcorollary}{Corollary}
\theoremstyle{definition}

\newtheorem*{unnumberedexample}{Example}
\theoremstyle{remark}

\newtheorem*{unnumberedremark}{Remark}
\theoremstyle{plain}

\newcommand{\C}{\mathbb C}

\newcommand{\dd}{\,d}
\newcommand{\Hol}{\operatorname{Hol}}
\newcommand{\dist}{\operatorname{dist}}
\newcommand{\Tr}{\operatorname{Tr}}
\newcommand{\rank}{\operatorname{rank}}
\newcommand{\avg}{\mathop{\fint}}
\newcommand{\Q}{\mathcal Q}
\newcommand{\V}{\mathcal V}
\newcommand{\Schatten}{\mathcal S}
\newcommand{\Acal}{\mathcal A}
\newcommand{\Dcal}{\mathcal D}

\title[Generalized Volterra companion operators]
{Generalized Volterra Companion Operators on Bergman Spaces over
Convex Domains of Finite Type}

\author{Jianxiang Dong}

\author{Chunxu Xu}


\hypersetup{
  pdftitle={Generalized Volterra Companion Operators on Bergman Spaces over Convex Domains of Finite Type},
  pdfauthor={Jianxiang Dong and Chunxu Xu},
  pdfsubject={Generalized Volterra companion operators on finite-type Bergman spaces},
  pdfkeywords={Generalized Volterra companion operator, Bergman space, convex finite-type domain, derivative Carleson measure, Schatten class}
}

\begin{document}

\begin{abstract}
We study generalized Volterra companion operators on Bergman spaces
over smoothly bounded convex domains of finite type.  A derivative
Carleson embedding gives boundedness and compactness criteria between
reflexive Bergman spaces.  When the target exponent is smaller than the
source exponent, boundedness already implies compactness.  In the
other range, we also obtain an essential-norm formula.  Local masses
on McNeal polydiscs describe these criteria.  Their norm estimates
require an extra weight because radial derivatives vanish at the
origin.  On the Hilbert Bergman space, we characterize Schatten-class
membership at and above the Hilbert--Schmidt threshold and prove a
sufficient condition below it.  We also give a Hilbert--Schmidt kernel
test.  For bounded symbols and self-maps with relatively compact image,
singular values decay exponentially in a power of their index.  An
ellipsoid example gives a sharp power law governed by boundary type
and dimension.  The results extend to positive radial shifts.
\end{abstract}

\maketitle

\section{Introduction}

For holomorphic $f$ and $g$ on the unit disc, the product formula
splits into two integral operators:
\[
 J_gf(z)=\int_0^z f(\zeta)g'(\zeta)\,\dd\zeta,
 \qquad
 I_gf(z)=\int_0^z f'(\zeta)g(\zeta)\,\dd\zeta.
\]
Thus $fg-f(0)g(0)=J_gf+I_gf$.  The two operators place the
derivative on different factors.  We consider a radial version of
the companion $I_g$ on higher-dimensional domains.

We work on a bounded convex domain $\Omega\subset\C^n$ with smooth
boundary of finite type.  Translate it so that $0\in\Omega$.  Write
$dv$ for Lebesgue volume and $\delta(z)=\dist(z,\partial\Omega)$.
For $1<p<\infty$, let
$A^p(\Omega)=\Hol(\Omega)\cap L^p(\Omega,dv)$ with its usual norm.
Convexity ensures that every radial segment from $0$ lies in $\Omega$.

Put $Rf(z)=\sum_{k=1}^n z_k\partial_k f(z)$ and let
$\mathcal H(\Omega)=\Hol(\Omega,\Omega)$.  If
$\varphi\in\mathcal H(\Omega)$ fixes $0$, define
\[
 C_{g,\varphi}f(z)
 =\int_0^1 Rf(\varphi(tz))g(tz)\,\frac{\dd t}{t}.
\]
The integrand is locally bounded near $t=0$.  Indeed,
$\varphi(tz)=O(t)$ and $Rf(0)=0$.  Differentiation along the
radial segment gives
\[
 C_{g,\varphi}f(0)=0,
 \qquad R(C_{g,\varphi}f)(z)=g(z)Rf(\varphi(z)).
\]
For $\varphi=\operatorname{id}$, write $C_g=C_{g,\operatorname{id}}$.
On the disc, $C_g=I_g$.  The self-map in $C_{g,\varphi}$ therefore
acts on the differentiated function.

Aleman and Siskakis studied Volterra integrals on Hardy and Bergman
spaces \cite{AlemanSiskakis1995,AlemanSiskakis1997}.  Further
Hardy-space results appear in
\cite{SiskakisZhao1999,AlemanCima2001,Rattya2007}; see also the
survey \cite{Siskakis2006}.  Hu studied extended Ces\`aro operators
on the ball \cite{Hu2004}.  Pau and Pel\'aez considered rapidly
decreasing Bergman weights \cite{PauPelaez2010,PauPelaez2012}.

Composition has also been studied in other analytic spaces.
Mengestie treated Volterra--composition products and generalized
companions on Fock spaces
\cite{Mengestie2014JGA,Mengestie2016PA,Mengestie2016BJMA}.
Arroussi, Gissy and Virtanen obtained boundedness and compactness
criteria for generalized Volterra operators on large Bergman spaces
of the disc \cite{ArroussiGissyVirtanen2023}.  Almoka, Arroussi and
Virtanen later treated Schatten classes in that setting
\cite{AlmokaArroussiVirtanen2025}.  Abkar studied related radial
composition--differentiation operators on the ball and polydisc
\cite{Abkar2024}.

The geometry of a finite-type convex domain is different.  Near a
weakly pseudoconvex boundary point, tangential and normal scales
can decay at different rates.  McNeal's polydiscs capture these
scales \cite{McNeal1992,McNeal1994}.  The Bergman projection
estimates of McNeal and Stein provide an analytic tool
\cite{McNealStein1994}.  Jasiczak \cite{Jasiczak2010} and Li, Liu
and Wang \cite{LiLiuWang2024} developed Carleson measure theory
using this local geometry.

For Schatten questions, Luecking's disc theorem provides an earlier
model \cite{Luecking1987}.  Xiao, Yang and Yuan studied ordinary
Carleson embeddings and positive Toeplitz operators on convex
finite-type domains \cite{XiaoYangYuan2026}.  Their Toeplitz
criterion covers Schatten exponents at least one.  Li, Long and
Wang addressed smaller exponents for ordinary Toeplitz operators
and weighted composition operators \cite{LiLongWang2026}.
Both papers concern embeddings that evaluate $f$, rather than $Rf$.
Their Toeplitz forms therefore have a different reproducing vector.
In particular, their small-exponent results do not supply a
small-exponent characterization for the companion studied here.

Our recent preprint treats absolute summability for the usual
Volterra--composition operator on convex finite-type domains
\cite{DongXu2026Abs}.  That operator satisfies
$R(V_{h,\varphi}f)=(f\circ\varphi)Rh$.  Here we instead study
$R(C_{g,\varphi}f)=g(Rf\circ\varphi)$.  The derivative is on $f$.
The preprint supplies the radial Littlewood--Paley formula used
below.  It does not give the derivative embedding required here.

The ordinary embeddings above measure $f$.  The companion requires
an embedding of $\delta Rf$.  On the Hilbert space, the associated
positive form is
\[
 \mathfrak s_\mu(f,f)
 =\int_\Omega |Rf(w)|^2\delta(w)^2\,\dd\mu(w).
\]
Its reproducing vectors involve derivatives of the Bergman kernel.
Ordinary Toeplitz criteria therefore do not directly give its
Schatten criterion.  The pull-back measure can also be singular.
In the ellipsoid example below, it is carried by a complex
hypersurface.
The derivative embedding theorem is the main new step.  It covers
both $p\le q$ and $q<p$ and accepts locally finite singular measures.
The latter range requires derivative atoms on separated boundary
cells.  For $s\ge2$, its Hilbert-space form also leads to a trace
criterion for the positive derivative form.  We make no claim of a
general necessary criterion for $0<s<2$.

There is a second issue for norm estimates.  The identity $Rf(0)=0$
makes mass at the origin invisible to the companion.  For
$\Omega=\mathbb D$, $g=1$, and $\varphi_\varepsilon(z)=\varepsilon z$,
\[
 C_{1,\varphi_\varepsilon}f(z)
 =f(\varepsilon z)-f(0),\qquad
 \|C_{1,\varphi_\varepsilon}\|_{A^2\to A^2}=\varepsilon.
\]
Yet the unweighted local mass near $0$ does not tend to zero as
$\varepsilon\to0$.  Weighting that mass by $|w|^q$ gives the
correct norm scale.  The weight has no effect on the boundary
criteria.  See \cref{subsec:examples-sharpness} for the calculation.

We first prove a derivative Carleson embedding theorem for
locally finite measures.  A radial Littlewood--Paley estimate then
reduces the companion to this embedding.  When $p\le q$,
derivative peak functions give the necessary condition and the
essential-norm lower bound.  When $q<p$, projection atoms and
Khinchine's inequality yield an integrability criterion.  In this
range, boundedness implies compactness.  The atom construction
uses convexity of $\log K(z,z)$ on convex domains
\cite{Xiong2026}.

For Schatten classes, we compare $C_{g,\varphi}^*C_{g,\varphi}$
with the positive form $\mathfrak s_\mu$.  A derivative trace
criterion gives an equivalence when $s\ge2$.  For $0<s<2$, we
prove a sufficient condition by local polynomial approximation.
On a convex ellipsoid, the singular values can be computed directly.
Their decay depends on both boundary type and dimension.

We use McNeal polydiscs.  For small $\rho>0$ let $\Q_\rho(z)$ be the
McNeal polydisc centered at $z$ and put
\[
 \V_\rho(z)=v(\Q_\rho(z)),
 \qquad
 \dd\lambda_\Omega(z)=\frac{\dd v(z)}{\V_\rho(z)}.
\]
Different small values of $\rho$ give equivalent conditions.  Define
\[
\int_\Omega h(w)\delta(w)^q\,\dd\mu_{g,\varphi,q}(w)
=
\int_{\{\varphi(z)\ne0\}}
h(\varphi(z))|g(z)|^q\delta(z)^q\,\dd v(z),
\qquad h\ge0 .
\]
The set $\varphi^{-1}(\{0\})$ is omitted.  It gives no term, since
$Rf(0)=0$.
Equivalently,
\[
 \mu_{g,\varphi,q}(E)
 =
 \int_{\varphi^{-1}(E\setminus\{0\})}
 |g(z)|^q
 \left(\frac{\delta(z)}{\delta(\varphi(z))}\right)^q
 \dd v(z).
\]
Set
\[
 \Acal_{g,\varphi,q,\rho}(z)
 =
 \left(
 \frac{\mu_{g,\varphi,q}(\Q_\rho(z))}{\V_\rho(z)}
 \right)^{1/q}
\]
and
\[
 \Dcal_{g,\varphi,p,q,\rho}(z)
 =
 \V_\rho(z)^{1/q-1/p}\Acal_{g,\varphi,q,\rho}(z).
\]
For the norm formulas, replace $\mu_{g,\varphi,q}$ by
$|w|^q\mu_{g,\varphi,q}$.  Denote the resulting functions by
$\Acal^{\circ}_{g,\varphi,q,\rho}$ and
$\Dcal^{\circ}_{g,\varphi,p,q,\rho}$.  These agree with the original
averages up to constants near the boundary.  They differ on compact
sets containing the origin.
Then
\[
\begin{aligned}
 \|C_{g,\varphi}f\|_{A^q}^q
 &\asymp
 \int_\Omega |Rf(\varphi(z))|^q|g(z)|^q\delta(z)^q\,\dd v(z)\\
 &=
 \int_\Omega |Rf(w)|^q\delta(w)^q\,\dd\mu_{g,\varphi,q}(w).
\end{aligned}
\]
The main theorems assume that this measure is locally finite.
The local averages are then well defined on compact subsets.

The same reduction holds for every nonsingular radial shift
\[
 C^{(\alpha)}_{g,\varphi}f(z)
 =
 \int_0^1 Rf(\varphi(tz))g(tz)t^{\alpha-1}\,\dd t,
 \qquad \alpha>0,
\]
because
\[
 (R+\alpha)C^{(\alpha)}_{g,\varphi}f=g\,Rf\circ\varphi .
\]
The case $\alpha=0$ is the operator $C_{g,\varphi}$ introduced above.
Thus we state the main conclusions for $\alpha\ge0$, with the convention
$C^{(0)}_{g,\varphi}=C_{g,\varphi}$.  We write
$\Schatten_s$ for the Schatten class of order $s$.
All operators are first defined on $\Hol(\overline\Omega)$.  When
boundedness holds, the same notation denotes the unique bounded
extension to the indicated Bergman space.

\begin{unnumberedtheorem}
Let $\alpha\ge0$, $g\in\Hol(\Omega)$, and
$\varphi\in\mathcal H(\Omega)$ with $\varphi(0)=0$.
\begin{enumerate}[label=\textup{(\roman*)}]
\item Let $1<p\le q<\infty$ and assume that
$\mu_{g,\varphi,q}$ is locally finite.  Then
$C^{(\alpha)}_{g,\varphi}:A^p(\Omega)\to A^q(\Omega)$ is bounded if and
only if
\[
 \Dcal_{g,\varphi,p,q,\rho}\in L^\infty(\Omega).
\]
Moreover,
\[
 \|C^{(\alpha)}_{g,\varphi}\|_{A^p\to A^q}
 \asymp_\alpha
 \|\Dcal^{\circ}_{g,\varphi,p,q,\rho}\|_{L^\infty}.
\]
The same operator is compact if and only if
\[
 \Dcal_{g,\varphi,p,q,\rho}(z)\to0,
 \qquad z\to\partial\Omega .
\]
If it is bounded, then
\[
 \|C^{(\alpha)}_{g,\varphi}\|_{e,A^p\to A^q}
 \asymp_\alpha
 \limsup_{z\to\partial\Omega}\Dcal_{g,\varphi,p,q,\rho}(z).
\]
Here $\|\cdot\|_e$ denotes the essential norm.

\item Let $1<q<p<\infty$, assume that $\mu_{g,\varphi,q}$ is locally
finite, and put
\[
 r=\frac{pq}{p-q}.
\]
Then the following are equivalent:
\[
\begin{aligned}
&C^{(\alpha)}_{g,\varphi}:A^p(\Omega)\to A^q(\Omega)
  \text{ is bounded},\\
&C^{(\alpha)}_{g,\varphi}:A^p(\Omega)\to A^q(\Omega)
  \text{ is compact},\\
&\Dcal_{g,\varphi,p,q,\rho}\in L^r(\Omega,d\lambda_\Omega).
\end{aligned}
\]
Moreover,
\[
 \|C^{(\alpha)}_{g,\varphi}\|_{A^p\to A^q}
 \asymp_\alpha
 \|\Dcal^{\circ}_{g,\varphi,p,q,\rho}\|_{L^r(d\lambda_\Omega)} .
\]

\item Let $2\le s<\infty$ and assume that $\mu_{g,\varphi,2}$ is locally
finite.  Then
\[
 C^{(\alpha)}_{g,\varphi}\in \Schatten_s(A^2(\Omega))
 \Longleftrightarrow
 \Acal_{g,\varphi,2,\rho}\in L^s(\Omega,d\lambda_\Omega).
\]
Moreover,
\[
 \|C^{(\alpha)}_{g,\varphi}\|_{\Schatten_s}^s
 \asymp_\alpha
 \int_\Omega \Acal^{\circ}_{g,\varphi,2,\rho}(z)^s\,\dd\lambda_\Omega(z).
\]
\end{enumerate}
\end{unnumberedtheorem}

The proof of equivalence uses the positive operator
$C_{g,\varphi}^*C_{g,\varphi}$.  Its Schatten exponent is
$s/2\ge1$ when $s\ge2$.  For $0<s<2$, the condition
$\Acal^{\circ}_{g,\varphi,2,\rho}\in L^s(d\lambda_\Omega)$ is
sufficient by \cref{cor:small-schatten-sufficient}.  We do not prove
necessity in this range.  The ellipsoid example gives an exact
criterion for every $s>0$.

If $g\in H^\infty(\Omega)$ and $\varphi(\Omega)$ lies in a compact
subset of $\Omega$, then $C^{(\alpha)}_{g,\varphi}:A^p\to A^q$
is compact for all $1<p,q<\infty$.  On $A^2$, its singular values
satisfy $s_j(C^{(\alpha)}_{g,\varphi})\le C\exp(-c j^{1/n})$.
Thus $C^{(\alpha)}_{g,\varphi}\in\Schatten_s$ for every $s>0$;
see \cref{cor:strict-self-map}.

For $\varphi=\operatorname{id}$ the measure becomes
\[
 \dd\mu_{g,\operatorname{id},q}=|g|^q\,\dd v .
\]
Hence the preceding results give sharp tests for the companion
operator $C_g$.  In particular, if $p<q$, boundedness of
$C_g:A^p\to A^q$ forces $g=0$.  If $p=q$, boundedness is equivalent to
$g\in H^\infty(\Omega)$, and compactness forces $g=0$.  Here
$H^\infty(\Omega)$ is the space of bounded holomorphic functions.
On $A^2(\Omega)$, $C_g$ belongs to a Schatten class only if $g=0$.
A composition map with relatively compact image behaves differently:
its pull-back measure is supported away from the boundary.

Section~2 develops the geometric and radial estimates.  Section~3
proves the derivative embedding theorem and applies it to boundedness,
compactness, and essential norms.  Section~4 studies Schatten classes
and compact-range singular values.

Constants denoted by $C$ may change from line to line.  They depend
only on the domain, the exponents, and fixed geometric parameters,
unless stated otherwise.
We write $A\asymp B$ if both $A\le CB$ and $B\le CA$ hold with such
constants.  We write $\avg_E$ for normalized integration over a set
$E$ of positive volume.

\section{Preparation}

We fix the geometric notation and the radial estimates used below.
The final subsection constructs tests for derivative evaluations.

\subsection{McNeal polydiscs}

We use the geometry of McNeal for bounded convex finite-type domains
\cite{McNeal1992,McNeal1994}.  If $\rho>0$ is small and $z\in\Omega$,
then $\Q_\rho(z)$ denotes a McNeal polydisc.  Its volume is
\[
 \V_\rho(z)=v(\Q_\rho(z)).
\]
If $w\in\Q_\rho(z)$, then
\[
 \V_\rho(w)\asymp \V_\rho(z),
 \qquad
 \delta(w)\asymp\delta(z),
\]
where
\[
 \delta(z)=\dist(z,\partial\Omega).
\]
For each small $\rho$ there is a McNeal lattice
$\{a_j\}_{j\ge1}\subset\Omega$ such that
\[
 \Omega=\bigcup_j \Q_\rho(a_j).
\]
We write
\[
 Q_j=\Q_\rho(a_j),\qquad V_j=\V_\rho(a_j).
\]
The smaller polydiscs are separated and fixed enlargements have bounded
overlap.

For non-negative functions which are essentially constant on small
polydiscs,
\[
 \int_\Omega F(z)\,\dd\lambda_\Omega(z)
 \asymp
 \sum_j F(a_j),
 \qquad
 \dd\lambda_\Omega(z)=\frac{\dd v(z)}{\V_\rho(z)}.
\]
In particular, for a positive measure $\mu$ and $0<t<\infty$,
\[
 \int_\Omega
 \left(\frac{\mu(\Q_\rho(z))}{\V_\rho(z)}\right)^t
 \dd\lambda_\Omega(z)
 \asymp
 \sum_j
 \left(\frac{\mu(Q_j)}{V_j}\right)^t .
\]

Let $K$ be the Bergman kernel and set
\[
 k_a(z)=\frac{K(z,a)}{K(a,a)^{1/2}} .
\]
McNeal's estimates give
\[
 K(a,a)\asymp \V_\rho(a)^{-1}.
\]
For a smaller polydisc,
\[
 |k_a(z)|^2\ge \frac{c}{\V_\rho(a)},
 \qquad z\in\Q_{\epsilon\rho}(a).
\]
We use local kernel estimates and bounded overlap of enlarged
polydiscs $Q_j^*$.  A normalized kernel need not decay pointwise
along a radial path toward the boundary.  This occurs even in the
disc.  The tests below use integral estimates.

\begin{lemma}\label{lem:radius-independent-companion}
Let $0<\rho_1,\rho_2$ be small.  Let $\alpha>0$ and $0<t<\infty$.
For a positive Borel measure $\mu$ put
\[
 A_\rho^\alpha\mu(z)
 =
 \frac{\mu(\Q_\rho(z))}{\V_\rho(z)^\alpha}.
\]
Then the following conditions do not depend on the chosen small radius:
\[
 \sup_{z\in\Omega} A_\rho^\alpha\mu(z)<\infty,
\]
\[
 \lim_{z\to\partial\Omega}A_\rho^\alpha\mu(z)=0,
\]
and
\[
 A_\rho^\alpha\mu\in L^t(\Omega,d\lambda_{\Omega,\rho}),
 \qquad
 d\lambda_{\Omega,\rho}(z)=\frac{\dd v(z)}{\V_\rho(z)}.
\]
\end{lemma}

\begin{proof}
Each $\Q_{\rho_1}(z)$ is covered by a bounded number of polydiscs
$\Q_{\rho_2}(b_m)$ with
\[
 b_m\in\Q_{C\rho_1}(z),
 \qquad
 \V_{\rho_2}(b_m)\asymp \V_{\rho_1}(z).
\]
Thus
\[
\begin{aligned}
 A_{\rho_1}^\alpha\mu(z)
 &=
 \frac{\mu(\Q_{\rho_1}(z))}{\V_{\rho_1}(z)^\alpha}       \\
 &\le
 C\sum_m
 \frac{\mu(\Q_{\rho_2}(b_m))}{\V_{\rho_2}(b_m)^\alpha}
 \le
 C\sup_{w\in\Q_{C\rho_1}(z)}A_{\rho_2}^\alpha\mu(w).
\end{aligned}
\]
The same estimate with $\rho_1$ and $\rho_2$ interchanged gives the
supremum statement.  Since $w\in\Q_{C\rho_1}(z)$ implies
$\delta(w)\asymp\delta(z)$, it also gives the vanishing statement.

Let $\{a_j\}$ be a $\rho_1$-lattice and $\{b_k\}$ a $\rho_2$-lattice.
The same covering argument and bounded overlap give
\[
 \sum_j
 \left(
 \frac{\mu(\Q_{\rho_1}(a_j))}{\V_{\rho_1}(a_j)^\alpha}
 \right)^t
 \le
 C
 \sum_k
 \left(
 \frac{\mu(\Q_{\rho_2}(b_k))}{\V_{\rho_2}(b_k)^\alpha}
 \right)^t .
\]
The reverse inequality is the same.  This proves the $L^t$ statement.
\end{proof}

For $g\in\Hol(\Omega)$ and $1\le q<\infty$ we also use the local
average
\[
 \Acal_{g,q,\rho}(z)
 =
 \left(
 \frac1{\V_\rho(z)}
 \int_{\Q_\rho(z)} |g(w)|^q\,\dd v(w)
 \right)^{1/q}.
\]

\begin{lemma}\label{lem:local-radial-cauchy}
Let $1\le q<\infty$.  Let $Q_j=\Q_\rho(a_j)$ be a McNeal lattice and
let $Q_j^*$ be a fixed enlargement.  Then
\[
 |Rf(z)|^q\delta(z)^q
 \le
 \frac{C}{V_j}
 \int_{Q_j^*}|f(w)|^q\,\dd v(w),
 \qquad z\in Q_j ,
\]
for every $f\in\Hol(\Omega)$.
\end{lemma}

\begin{proof}
Fix $z\in Q_j$.  Let $\tau_1(z),\ldots,\tau_n(z)$ be the McNeal
distinguished radii at $z$.  Cauchy's formula on the distinguished
polydisc gives
\[
 \left|\frac{\partial f}{\partial \zeta_\ell}(z)\right|^q
 \le
 \frac{C}{\tau_\ell(z)^q\,\V_\rho(z)}
 \int_{\Q_{c\rho}(z)}|f(w)|^q\,\dd v(w).
\]
The radial vector decomposes in the distinguished frame as
\[
 z=\sum_{\ell=1}^n z_\ell^* e_\ell(z).
\]
Since $\Omega$ is bounded, $|z_\ell^*|\le C$.  The normal radius is
comparable to $\delta(z)$.  Each distinguished radius is at least
$c\delta(z)$, since $B(z,c\delta(z))\subset\Omega$.  Hence
\[
 \delta(z)|z_\ell^*|\le C\tau_\ell(z),
 \qquad 1\le \ell\le n .
\]
Thus
\[
 \delta(z)^q |Rf(z)|^q
 \le
 C\sum_{\ell=1}^n
 \tau_\ell(z)^q
 \left|\frac{\partial f}{\partial \zeta_\ell}(z)\right|^q .
\]
Combining the two estimates gives
\[
 |Rf(z)|^q\delta(z)^q
 \le
 \frac{C}{\V_\rho(z)}
 \int_{\Q_{c\rho}(z)}|f(w)|^q\,\dd v(w).
\]
If $z\in Q_j$, then
\[
 \Q_{c\rho}(z)\subset Q_j^*,
 \qquad
 \V_\rho(z)\asymp V_j .
\]
The result follows.
\end{proof}

\subsection{Radial formulas}

Since $0\in\Omega$ and $\Omega$ is convex, each segment
\[
 [0,z]=\{tz:0\le t\le1\}
\]
lies in $\Omega$.  If $\xi\in\partial\Omega$, then the radial direction
is uniformly transverse to the boundary.  Indeed, if $\nu(\xi)$ is the
outward unit normal, then
\[
 \operatorname{Re}\langle \xi,\nu(\xi)\rangle>0.
\]
The boundary is compact.  Hence
\[
 c_0\le \operatorname{Re}\langle \xi,\nu(\xi)\rangle\le C_0,
 \qquad \xi\in\partial\Omega .
\]
It follows that
\[
 C^{-1}(1-r)\le \delta(r\xi)\le C(1-r),
 \qquad 0\le r<1,\quad \xi\in\partial\Omega .
\]

We use radial coordinates
\[
 z=r\xi,\qquad 0<r<1,\quad \xi\in\partial\Omega .
\]
Then
\[
 \dd v(r\xi)=J(r,\xi)\,\dd r\,\dd\sigma(\xi),
\]
where $d\sigma$ is surface measure on $\partial\Omega$, and $J$ is
smooth and positive away from the origin.  Near the boundary,
\[
 J(r,\xi)\asymp J(1,\xi).
\]
Estimates on compact subsets are absorbed in the constants.

\begin{lemma}\label{lem:radial-approximation}
Let $1\le p<\infty$ and $f\in A^p(\Omega)$.  For $0<r<1$ put
\[
 f_r(z)=f(rz).
\]
Then $f_r\in\Hol(\overline\Omega)$ and
\[
 \|f_r-f\|_{A^p}\to0,
 \qquad r\to1^- .
\]
Consequently, $\Hol(\overline\Omega)$ is dense in $A^p(\Omega)$.
\end{lemma}

\begin{proof}
Since $0\in\Omega$ and $\Omega$ is convex, $r\overline\Omega$ is a
compact subset of $\Omega$ for every $0<r<1$.  Hence $f_r$ is
holomorphic in a neighbourhood of $\overline\Omega$.  Extend $f$ by
$0$ to an $L^p(\C^n)$ function.  The dilation operators are strongly
continuous on $L^p(\C^n)$.  Therefore
\[
 \int_\Omega |f(rz)-f(z)|^p\,\dd v(z)
 \le
 \int_{\C^n}|f(rz)-f(z)|^p\,\dd v(z)
 \to0 .
\]
This proves the assertion.
\end{proof}

The next lemma is the radial Littlewood--Paley formula on the
codimension-one subspace of functions vanishing at the origin.
The general formula for functions based at any interior point also
appears in \cite{DongXu2026Abs}.  We give a direct proof for the
range needed here.

\begin{lemma}\label{lem:radial-lp-companion}
Let $1<q<\infty$.  If $h\in\Hol(\Omega)$ and $h(0)=0$, then
\[
 \|h\|_{A^q}^q
 \asymp
 \int_\Omega |Rh(z)|^q\delta(z)^q\,\dd v(z).
\]
\end{lemma}

\begin{proof}
Put $H=Rh$ and $d=2n$.  Since $h(0)=H(0)=0$,
\[
 h(z)=\int_0^1H(tz)\,\frac{\dd t}{t}.
\]
First control the inner part of the domain.  Choose $t_0>0$ so
small that $t\bigl(\tfrac23\Omega\bigr)$ lies in a fixed ball
about $0$ for $0<t<t_0$.  Cauchy's estimate on a larger fixed ball
gives
\[
 \sup_{z\in(2/3)\Omega}|H(tz)|
 \le Ct\,\|H\|_{L^q((3/4)\Omega)},\qquad 0<t<t_0.
\]
For $t_0\le t\le1$, Minkowski's inequality and the change of
variables $w=tz$ give
\[
 \|H(t\,\cdot)\|_{L^q((2/3)\Omega)}
 \le t^{-d/q}\|H\|_{L^q((3/4)\Omega)}.
\]
Integrating with respect to $\dd t/t$ yields
\begin{equation}\label{eq:lp-interior-companion}
 \|h\|_{L^q((2/3)\Omega)}
 \le C\|H\|_{L^q((3/4)\Omega)}.
\end{equation}
The submean inequality on fixed balls around
$\partial((1/2)\Omega)$ now gives
\begin{equation}\label{eq:lp-inner-trace-companion}
 \int_{\partial\Omega}|h(\xi/2)|^q\,\dd\sigma(\xi)
 \le C\|H\|_{L^q((3/4)\Omega)}^q.
\end{equation}

Fix $\xi\in\partial\Omega$.  Since
\[
 \frac{\dd}{\dd r}h(r\xi)=\frac{H(r\xi)}r,
\]
we may write, for $1/2\le r<1$,
\[
 h(r\xi)=h(\xi/2)+\int_{1/2}^rH(t\xi)\,\frac{\dd t}{t}.
\]
Put $G_\xi(r)=\int_{1/2}^r|H(t\xi)|\,\dd t/t$.  For
$1/2<R<1$, integration by parts followed by H\"older's inequality
gives
\[
 \int_{1/2}^R G_\xi(r)^q\,\dd r
 \le q\left(\int_{1/2}^R G_\xi(r)^q\,\dd r\right)^{(q-1)/q}
 \left(\int_{1/2}^R(1-r)^qG_\xi'(r)^q\,\dd r\right)^{1/q}.
\]
Divide when the first integral is nonzero.  Since $r\ge1/2$,
letting $R\uparrow1$ gives
\[
 \int_{1/2}^1G_\xi(r)^q\,\dd r
 \le C_q\int_{1/2}^1(1-r)^q|H(r\xi)|^q\,\dd r.
\]
The radial Jacobian is
$J(r,\xi)=r^{d-1}J(1,\xi)$.  On $[1/2,1)$, it is comparable
to $J(1,\xi)$; also $\delta(r\xi)\asymp1-r$.
Integrate in $\xi$ and use \eqref{eq:lp-inner-trace-companion}.
Together with \eqref{eq:lp-interior-companion}, this yields
\[
 \|h\|_{A^q}^q
 \le
 C\int_\Omega |H(z)|^q\delta(z)^q\,\dd v(z).
\]
Here $\delta$ is bounded below on $(3/4)\Omega$.

For the reverse estimate, Cauchy's estimate in a Euclidean ball
$B(z,c\delta(z))\subset\Omega$ gives
\[
 |Rh(z)|^q\delta(z)^q
 \le
 C\avg_{B(z,c\delta(z))}|h(w)|^q\,\dd v(w).
\]
Using a Whitney covering with bounded overlap, we obtain
\[
 \int_\Omega |Rh(z)|^q\delta(z)^q\,\dd v(z)
 \le
 C\int_\Omega |h(w)|^q\,\dd v(w).
\]
\end{proof}

\begin{lemma}\label{lem:companion-identity}
Let $g\in\Hol(\Omega)$ and let
$\varphi\in\mathcal H(\Omega)$ satisfy $\varphi(0)=0$.  For
$f\in\Hol(\overline\Omega)$, that is, for $f$ holomorphic in a
neighbourhood of $\overline\Omega$, put
\[
 C_{g,\varphi}f(z)
 =
 \int_0^1 Rf(\varphi(tz))g(tz)\,\frac{\dd t}{t}.
\]
Then $C_{g,\varphi}f\in\Hol(\Omega)$,
\[
 C_{g,\varphi}f(0)=0,
 \qquad
 R(C_{g,\varphi}f)(z)=g(z)Rf(\varphi(z)).
\]
\end{lemma}

\begin{proof}
The integrand is holomorphic in $z$.  Since $\varphi(0)=0$,
\[
 \varphi(tz)=O(t)
\]
on compact subsets.  Also $Rf(0)=0$.  Hence, near $t=0$,
\[
 Rf(\varphi(tz))=O(t),
\]
uniformly for $z$ in compact subsets.  Thus the integral is well
defined, since $Rf(\varphi(tz))g(tz)/t=O(1)$.  It is holomorphic by
uniform convergence on compact subsets.  Moreover,
\[
 C_{g,\varphi}f(rz)
 =
 \int_0^r Rf(\varphi(sz))g(sz)\,\frac{\dd s}{s}.
\]
Hence
\[
 r\frac{\dd}{\dd r}C_{g,\varphi}f(rz)
 =
 Rf(\varphi(rz))g(rz).
\]
Putting $r=1$ gives the identity.
\end{proof}

\begin{unnumberedremark}
Let $\alpha>0$ and define
\[
 C^{(\alpha)}_{g,\varphi}f(z)
 =
 \int_0^1 Rf(\varphi(tz))g(tz)t^{\alpha-1}\,\dd t .
\]
The integral is proper at $t=0$.  Indeed, $\varphi(0)=0$ and $Rf(0)=0$
give $Rf(\varphi(tz))=O(t)$ on compact subsets, so the integrand is
$O(t^\alpha)$.
Then the same proof gives
\[
 (R+\alpha)C^{(\alpha)}_{g,\varphi}f
 =
 g\,Rf\circ\varphi .
\]
The main operator is the limiting case $\alpha=0$.
\end{unnumberedremark}

\begin{proposition}\label{prop:companion-reduction}
Let $1<q<\infty$, $g\in\Hol(\Omega)$, and
$\varphi\in\mathcal H(\Omega)$ with $\varphi(0)=0$.  Define
$\mu_{g,\varphi,q}$ by
\[
 \int_\Omega h(w)\delta(w)^q\,\dd\mu_{g,\varphi,q}(w)
 =
 \int_{\{\varphi(z)\ne0\}}
 h(\varphi(z))|g(z)|^q\delta(z)^q\,\dd v(z),
 \qquad h\ge0 .
\]
For $f\in\Hol(\overline\Omega)$,
\[
 \|C_{g,\varphi} f\|_{A^q}^q
 \asymp
 \int_\Omega |Rf(w)|^q\delta(w)^q\,
 \dd\mu_{g,\varphi,q}(w).
\]
\end{proposition}

\begin{proof}
By \cref{lem:companion-identity},
\[
 C_{g,\varphi} f(0)=0,
 \qquad
 R(C_{g,\varphi} f)(z)=g(z)Rf(\varphi(z)).
\]
Using \cref{lem:radial-lp-companion} with $h=C_{g,\varphi} f$, we get
\[
\begin{aligned}
 \|C_{g,\varphi} f\|_{A^q}^q
 &\asymp
 \int_\Omega |R(C_{g,\varphi} f)(z)|^q\delta(z)^q\,\dd v(z)\\
 &=
 \int_\Omega |Rf(\varphi(z))|^q|g(z)|^q\delta(z)^q\,\dd v(z)\\
 &=
 \int_\Omega |Rf(w)|^q\delta(w)^q\,\dd\mu_{g,\varphi,q}(w).
\end{aligned}
\]
\end{proof}

\subsection{Derivative peak functions}

The next lemma supplies test functions for necessity.  It is the
derivative analogue of the normalized kernel test.

\begin{lemma}\label{lem:mcneal-peaks}
Let $1<p<\infty$.  There are constants $\epsilon>0$ and
$C>0$ such that, for every $a\in\Omega$, there is
$P_a\in\Hol(\Omega)$ with
\[
 \|P_a\|_{A^p}\le C,
\]
\[
 |P_a(z)|\ge C^{-1}\V_\rho(a)^{-1/p},
 \qquad z\in\Q_{\epsilon\rho}(a),
\]
\[
 |RP_a(z)|
 \le
 C\delta(a)^{-1}\V_\rho(a)^{-1/p},
 \qquad z\in\Q_{\epsilon\rho}(a).
\]
In particular, $P_a\to0$ uniformly on compact subsets as
$a\to\partial\Omega$.
\end{lemma}

\begin{proof}
Write $V=\V_\rho(a)$ and first take $a$ near the boundary.
Let $Q_j(a)$ denote the McNeal polydisc at scale
$2^j\rho\delta(a)$ and put $W_j=v(Q_j(a))$.
Stop when this scale reaches a fixed positive number.
The off-diagonal estimate of \cite{McNeal1994}, in the quasimetric
form used in \cite[Proposition~3.4]{KrantzLi1995}, and finite-type
volume growth give
\[
 \begin{gathered}
 |K(z,a)|\le C W_j^{-1}
 \quad(z\in Q_{j+1}(a)\setminus Q_j(a)),\\
 W_{j+1}\le C W_j,\qquad W_j/V\ge c2^{2j}.
 \end{gathered}
\]
The last inequality follows already from growth in the complex
normal direction.  Off the last polydisc the kernel is bounded
uniformly by its regularity away from the boundary diagonal.
For the normal affine function $L_a$ used in
\cref{lem:derivative-peaks}, the normal radius also gives
\[
 |L_a(z)|\le C2^{j+1}\quad(z\in Q_{j+1}(a)),
 \qquad V\le C\delta(a)^2.
\]

Choose an integer $m$ so that $2(mp-1)>p$ and $m-1/p\ge1$.
Set
\[
 P_a(z)=V^{-1/p}\left(\frac{K(z,a)}{K(a,a)}\right)^m.
\]
Since $K(a,a)\asymp V^{-1}$, on the $j$th shell
\[
 |P_a(z)|\le C V^{-1/p}(V/W_j)^m.
\]
The initial polydisc contributes a bounded amount to
$\int |L_aP_a|^p$.  On the remaining shells,
\[
 \begin{aligned}
 \int_{\bigcup_j(Q_{j+1}\setminus Q_j)}|L_aP_a|^p\,\dd v
 &\le C\sum_j 2^{jp}\frac{W_{j+1}}{V}
                  \left(\frac{V}{W_j}\right)^{mp}\\
 &\le C\sum_j 2^{-j[2(mp-1)-p]}\le C.
 \end{aligned}
\]
Off the last polydisc, $|L_a|\le C\delta(a)^{-1}$ and
$|K(z,a)/K(a,a)|\le CV$.  Thus the remaining integral is at most
\[
 C\delta(a)^{-p}V^{mp-1}\le C.
\]
The same bounds with $L_a$ omitted give $\|P_a\|_{A^p}\le C$.
The near-diagonal estimate gives
\[
 \left|\frac{K(z,a)}{K(a,a)}\right|\ge c,
 \qquad z\in\Q_{\epsilon\rho}(a).
\]
This proves the stated lower bound.  Cauchy's estimate on a
slightly larger polydisc gives
\[
 |RP_a(z)|\le C\delta(a)^{-1}V^{-1/p},
 \qquad z\in\Q_{\epsilon\rho}(a).
\]
For every compact $K\Subset\Omega$, off-diagonal regularity gives
$\sup_{z\in K}|K(z,a)|\le C_K$.  Consequently,
\[
 \sup_{z\in K}|P_a(z)|
 \le C_K V^{m-1/p}\le C_K\delta(a)^2.
\]
For $a$ in a fixed compact core, take $P_a=V^{-1/p}$ instead.
Then $\|P_a\|_{A^p}\le C$ and the two local estimates hold.
\end{proof}

\begin{lemma}\label{lem:derivative-peaks}
Let $1<p<\infty$.  There are constants $\epsilon>0$, $\eta>0$, and
$C>0$ with the following property.  For every $a\in\Omega$ with
$\delta(a)<\eta$ there is
$F_a\in A^p(\Omega)$ such that
\[
 \|F_a\|_{A^p}\le C,
\]
\[
 |RF_a(z)|\delta(z)\ge C^{-1}\V_\rho(a)^{-1/p},
 \qquad z\in\Q_{\epsilon\rho}(a),
\]
and, as $a\to\partial\Omega$,
\[
 F_a\to0
 \quad\text{uniformly on compact subsets of }\Omega .
\]
\end{lemma}

\begin{proof}
Let $\nu(a)$ be the complex normal direction at a nearest boundary
point to $a$.  By radial transversality, the radial vector has a normal
component bounded below.  Let
\[
 L_a(z)=\frac{\langle z-a,\nu(a)\rangle}{\delta(a)} .
\]
On $\Q_{\rho}(a)$ the function $L_a$ is uniformly bounded and
\[
 |RL_a(z)|\ge c\,\delta(a)^{-1},
 \qquad z\in\Q_{\epsilon\rho}(a).
\]
Take $P_a$ from \cref{lem:mcneal-peaks}.  Then
\[
 \|P_a\|_{A^p}\asymp1,\qquad
 |P_a(z)|\ge c\,\V_\rho(a)^{-1/p},
 \qquad z\in\Q_{\epsilon\rho}(a),
\]
and
\[
 |RP_a(z)|\le C\delta(a)^{-1}\V_\rho(a)^{-1/p}
 \quad\text{on }\Q_{\epsilon\rho}(a).
\]
Put
\[
 F_a(z)=c_1\,L_a(z)P_a(z).
\]
The annular estimate in the proof of \cref{lem:mcneal-peaks} gives
\[
 \int_\Omega |L_a(z)P_a(z)|^p\,\dd v(z)
 \le C.
\]
Fix a small $c_1>0$.  Then
\[
 \|F_a\|_{A^p}\le C.
\]
Also
\[
 R F_a=(RL_a)P_a+L_aRP_a .
\]
On $\Q_{\epsilon\rho}(a)$ we have
\[
 |(RL_a)P_a|
 \ge
 c\delta(a)^{-1}\V_\rho(a)^{-1/p}.
\]
Moreover,
\[
 |L_a(z)|\le C\epsilon,
 \qquad z\in\Q_{\epsilon\rho}(a),
\]
and hence
\[
 |L_aRP_a|
 \le
 C\epsilon\,\delta(a)^{-1}\V_\rho(a)^{-1/p}.
\]
Choose $\epsilon$ so small that the last term is at most one half of
the preceding lower bound.  Since $\delta(z)\asymp\delta(a)$ on
$\Q_{\epsilon\rho}(a)$, we get
\[
 |RF_a(z)|\delta(z)\ge C^{-1}\V_\rho(a)^{-1/p}
 \quad z\in\Q_{\epsilon\rho}(a).
\]
Finally, let $K\Subset\Omega$.  On $K$,
\[
 |L_a(z)|\le C_K\delta(a)^{-1}.
\]
Using the preceding choice of the peak power, we have
\[
 \sup_{z\in K}|P_a(z)|\le C_K\delta(a)^2=o(\delta(a)).
\]
Then
\[
 \sup_{z\in K}|F_a(z)|
 \le C_K\delta(a)^{-1}\sup_{z\in K}|P_a(z)|
 \to0 .
\]
The lemma is only used near the boundary.
\end{proof}

\begin{lemma}[Projection atoms for separated boundary cells]
\label{lem:projection-atoms}
Fix $1<p<\infty$.  After dividing a sufficiently fine boundary
lattice into finitely many separated sublattices, each sublattice
$\{a_j\}$ admits functions $G_j\in A^p(\Omega)$ and a fixed
$\epsilon>0$ such that
\[
 \left\|\sum_j c_jG_j\right\|_{A^p}
 \le C\left(\sum_j|c_j|^p\right)^{1/p},
 \qquad
 \delta(z)|RG_j(z)|\ge c V_j^{-1/p}
 \quad(z\in\Q_{\epsilon\rho}(a_j)).
\]
The sums are initially finite.  Also $G_j\to0$ locally uniformly
as $a_j\to\partial\Omega$.
\end{lemma}

\begin{proof}
Let $B$ be the ordinary Bergman projection.  It is bounded on
$L^p(\Omega)$ for $1<p<\infty$ by \cite{McNealStein1994}.
We first prove a lower bound for the derivative kernel.
Xiong proved the convexity of $z\mapsto\log K(z,z)$ on convex
domains \cite[Theorem~1.1]{Xiong2026}.  Write $a=r\xi$ with
$\xi\in\partial\Omega$ and put $d=1-r\asymp\delta(a)$.
Fix a large $L>1$.  If $d$ is small, put $b=(1-Ld)\xi$ and
$t=(1-Ld)/r$.  Then $b=ta$ and $t\ge1/2$.

Convexity and $0\in\Omega$ give $t\Q_\rho(a)\subset\Omega$.
Radial transversality gives $\delta(b)\ge cLd$.  Hence
$B(b,cLd/2)\subset\Omega$.  In the distinguished coordinates at
$a$, write the radii of $\Q_\rho(a)$ as
$\tau_1(a),\ldots,\tau_{n-1}(a)$ in the tangential directions and
$\tau_n(a)\asymp d$ in the normal direction.  The set
\[
 \tfrac12\bigl(t\Q_\rho(a)+B(b,cLd/2)\bigr)
 \subset\Omega
\]
contains a polydisc centered at $b$ with tangential radii at least
$c_\rho\tau_j(a)$ and normal radius at least $c_\rho Ld$.
This follows by taking coordinate polydiscs inside the two summands.
Its volume is therefore at least
\[
 c_\rho(Ld)^2\prod_{j<n}\tau_j(a)^2
 \asymp c_\rho L^2\V_\rho(a).
\]

The mean-value inequality on this polydisc gives
$K(b,b)\le C_\rho L^{-2}\V_\rho(a)^{-1}$.  McNeal's diagonal
estimate gives $K(a,a)\ge c_\rho\V_\rho(a)^{-1}$.  Therefore
\[
 \frac{K(a,a)}{K(b,b)}
 \ge c_\rho L^2
\]
uniformly in $a$ near the boundary.  Fix $L$ so large that the ratio
exceeds $2$.  Convexity along the radial segment bounds the
derivative at $r$ below by the secant slope between $b$ and $a$.
Hence
\[
 \frac{d}{dr}\log K(r\xi,r\xi)
 \ge\frac{\log 2}{(L-1)d}.
\]
The derivative on the left is
$2\operatorname{Re}(\xi\cdot\partial_zK(a,a))/K(a,a)$.
Also $r$ stays away from zero.  Hence
\begin{equation}\label{eq:diagonal-radial-kernel}
 \delta(a)|R_zK(a,a)|\ge cK(a,a)\ge c'\V_\rho(a)^{-1}.
\end{equation}
Apply the near-diagonal kernel estimate and Cauchy's inequality in
the second variable.  For a fixed, sufficiently small $\epsilon>0$,
\[
 \sup_{w\in\Q_{\epsilon\rho}(a)}
 \delta(a)|R_zK(a,w)-R_zK(a,a)|
 \le C\epsilon^\gamma\V_\rho(a)^{-1}
\]
for some $\gamma>0$.  Choose $\epsilon$ so that the right side is
less than half the lower bound in
\eqref{eq:diagonal-radial-kernel}.  It follows that
\[
 \delta(a)|R_zK(a,w)|\ge c\V_\rho(a)^{-1}
 \quad(w\in E_a:=\Q_{\epsilon\rho}(a)).
\]
These constants are uniform in the boundary centers.  Integrating
over $E_a$ gives
\begin{equation}\label{eq:local-derivative-mass}
 \delta(a)\int_{E_a}|R_zK(a,w)|\,\dd v(w)\ge c_\epsilon,
 \qquad v(E_a)\asymp\V_\rho(a).
\end{equation}
Color the lattice so the sets $E_{a_j}$ in each color are disjoint.
Define
\[
 u_j(w)=V_j^{-1/p}{\bf1}_{E_{a_j}}(w)
 \frac{\overline{R_zK(a_j,w)}}{|R_zK(a_j,w)|},
 \qquad G_j=Bu_j,
\]
with the quotient set equal to zero where its denominator vanishes.
Then
\[
 \left\|\sum_j c_jG_j\right\|_{A^p}
 \le C\left\|\sum_jc_ju_j\right\|_{L^p}
 \le C\left(\sum_j|c_j|^p\right)^{1/p}.
\]
Each $u_j$ lies in $L^2\cap L^p$.  We may therefore differentiate
its kernel integral.  By \eqref{eq:local-derivative-mass},
\[
 \delta(a_j)|RG_j(a_j)|
 =V_j^{-1/p}\delta(a_j)
   \int_{E_{a_j}}|R_zK(a_j,w)|\,\dd v(w)
 \ge c V_j^{-1/p}.
\]
The local Cauchy estimate bounds the oscillation of $\delta RG_j$
on $\Q_{\epsilon'\rho}(a_j)$ by $C\epsilon' V_j^{-1/p}$.
Choose $\epsilon'>0$ small.  The asserted lower bound follows.

As the cells leave every compact set, the disjointly supported
$u_j$ converge weakly to zero in $L^p$.  The operator $B$ is bounded
on $L^p$.  Point evaluations are bounded on $A^p$.  Hence
$G_j\to0$ uniformly on compact subsets.
\end{proof}

\begin{unnumberedlemma}
Let $1<p<\infty$.  If $\{f_m\}$ is bounded in $A^p(\Omega)$ and
$f_m\to0$ uniformly on compact subsets of $\Omega$, then $f_m$ converges
weakly to $0$ in $A^p(\Omega)$.  Hence $Kf_m\to0$ in norm for every
compact operator $K:A^p(\Omega)\to X$.
\end{unnumberedlemma}

\begin{proof}
The space $A^p(\Omega)$ is reflexive.  Every subsequence of
$\{f_m\}$ has a weakly convergent subsequence.  Point evaluations
are bounded, and $f_m\to0$ on compact subsets.  Hence every such
weak limit is zero.  It follows that $f_m$ converges weakly to zero.

Now let $K$ be compact.  Every subsequence of $\{Kf_m\}$ has a
norm-convergent subsequence.  Its limit must be zero by weak
convergence.  Thus $Kf_m\to0$ in norm.
\end{proof}

\section{Boundedness and Compactness}

We first treat derivative embeddings for arbitrary positive measures.
We then apply the result to the pull-back measure of a companion.

\subsection{Derivative Carleson embeddings}

For a positive locally finite Borel measure $\mu$ set
\[
 \Acal_{\mu,q,\rho}(z)
 =
 \left(
 \frac{\mu(\Q_\rho(z))}{\V_\rho(z)}
 \right)^{1/q}
\]
and
\[
 \Dcal_{\mu,p,q,\rho}(z)
 =
 \V_\rho(z)^{1/q-1/p}\Acal_{\mu,q,\rho}(z).
\]
For the symbol $g$ we use
\[
 \dd\mu_g(z)=|g(z)|^q\,\dd v(z).
\]
Then
\[
 \Acal_{g,q,\rho}=\Acal_{\mu_g,q,\rho},
 \qquad
 \Dcal_{g,p,q,\rho}=\Dcal_{\mu_g,p,q,\rho}.
\]
For the generalized companion operator we use
\[
 \int_\Omega h(w)\delta(w)^q\,\dd\mu_{g,\varphi,q}(w)
 =
 \int_{\{\varphi(z)\ne0\}}
 h(\varphi(z))|g(z)|^q\delta(z)^q\,\dd v(z),
 \qquad h\ge0 .
\]
The omitted set gives no contribution to $C_{g,\varphi}$, because
$Rf(0)=0$ for every holomorphic $f$.
Thus
\[
 \mu_{g,\varphi,q}(E)
 =
 \int_{\varphi^{-1}(E\setminus\{0\})}
 |g(z)|^q
 \left(\frac{\delta(z)}{\delta(\varphi(z))}\right)^q
 \dd v(z).
\]
We write
\[
 \Acal_{g,\varphi,q,\rho}
 =
 \Acal_{\mu_{g,\varphi,q},q,\rho},
 \qquad
 \Dcal_{g,\varphi,p,q,\rho}
 =
 \Dcal_{\mu_{g,\varphi,q},p,q,\rho}.
\]

The radial derivative vanishes at the origin.  To measure the norm (as
opposed to membership or boundary behaviour), put
\[
 \dd\nu_{\mu,q}(w)=|w|^q\,\dd\mu(w),\qquad
 \Dcal^{\circ}_{\mu,p,q,\rho}
 =\Dcal_{\nu_{\mu,q},p,q,\rho}.
\]
We write \(\Dcal^{\circ}_{g,\varphi,p,q,\rho}\) for the latter quantity
when \(\mu=\mu_{g,\varphi,q}\).  When \(q=2\), set
\[
 \Acal^{\circ}_{\mu,2,\rho}(z)
 =\left(\frac{\nu_{\mu,2}(\Q_\rho(z))}
 {\V_\rho(z)}\right)^{1/2},\qquad
 \Acal^{\circ}_{g,\varphi,2,\rho}
 =\Acal^{\circ}_{\mu_{g,\varphi,2},2,\rho}.
\]
Near the boundary, corrected and uncorrected averages are comparable.
On a fixed compact set, corrected averages measure the mass of
\(|w|^q\mu\).  For a locally finite measure, both choices give the
same membership and boundary-vanishing conditions.  Their norms
need not be comparable uniformly in \(\mu\).

\begin{theorem}\label{thm:derivative-carleson}
Let $1<p,q<\infty$ and let $\mu$ be a positive locally finite Borel
measure on $\Omega$ with $\mu(\{0\})=0$.  Consider
\[
 \int_\Omega |Rf(z)|^q\delta(z)^q\,\dd\mu(z)
 \le C\|f\|_{A^p}^q .
\]
Here and below $C$ in a norm comparison denotes the least admissible
constant in this inequality.
If $p\le q$, then this estimate holds for all
$f\in A^p(\Omega)$ if and only if
\[
 \sup_{z\in\Omega}\Dcal_{\mu,p,q,\rho}(z)<\infty .
\]
Moreover,
\[
 C^{1/q}\asymp
 \|\Dcal^{\circ}_{\mu,p,q,\rho}\|_{L^\infty}.
\]
If $q<p$ and
\[
 r=\frac{pq}{p-q},
\]
then the estimate holds if and only if
\[
 \Dcal_{\mu,p,q,\rho}\in L^r(\Omega,d\lambda_\Omega).
\]
Moreover,
\[
 C^{1/q}\asymp
 \|\Dcal^{\circ}_{\mu,p,q,\rho}\|_{L^r(d\lambda_\Omega)} .
\]
\end{theorem}

\begin{proof}
We first prove sufficiency.  Let $\{a_j\}$ be a McNeal lattice and put
$Q_j=\Q_\rho(a_j)$, $V_j=\V_\rho(a_j)$.  By
\cref{lem:local-radial-cauchy},
\[
 |Rf(z)|^q\delta(z)^q
 \le
 \frac{C}{V_j}\int_{Q_j^*}|f(w)|^q\,\dd v(w),
 \qquad z\in Q_j .
\]
Hence
\begin{equation}\label{eq:derivative-lattice-upper}
\begin{aligned}
 \int_\Omega |Rf|^q\delta^q\,\dd\mu
 &\le
 C\sum_j
 \frac{\mu(Q_j)}{V_j}
 \int_{Q_j^*}|f(w)|^q\,\dd v(w).
\end{aligned}
\end{equation}

Assume first that $p\le q$ and
\[
 M=\sup_j
 \frac{\mu(Q_j)}{V_j^{q/p}}<\infty .
\]
By the sub-mean estimate,
\[
 \int_{Q_j^*}|f|^q\,\dd v
 \le
 C V_j
 \left(
 \frac1{V_j^{**}}\int_{Q_j^{**}}|f|^p\,\dd v
 \right)^{q/p}.
\]
Thus
\[
\begin{aligned}
 \int_\Omega |Rf|^q\delta^q\,\dd\mu
 &\le
 C M
 \sum_j
 \left(
 \int_{Q_j^{**}} |f|^p\,\dd v
 \right)^{q/p}                                      \\
 &\le
 C M
 \left(
 \sum_j\int_{Q_j^{**}} |f|^p\,\dd v
 \right)^{q/p}
 \le
 C M\|f\|_{A^p}^q .
\end{aligned}
\]
This proves the unweighted sufficiency assertion.  The quantitative
comparison with $\Dcal^{\circ}_{\mu,p,q,\rho}$ is established at the
end of the proof by separating a fixed compact core from the boundary
collar.

Now assume $q<p$.  Put
\[
 c_j=\frac{\mu(Q_j)}{V_j^{q/p}} .
\]
By H\"older's inequality on $Q_j^*$,
\[
 \int_{Q_j^*}|f|^q\,\dd v
 \le
 C V_j^{1-q/p}
 \left(
 \int_{Q_j^*}|f|^p\,\dd v
 \right)^{q/p}.
\]
Using \eqref{eq:derivative-lattice-upper}, we get
\[
 \int_\Omega |Rf|^q\delta^q\,\dd\mu
 \le
 C\sum_j c_j
 \left(
 \int_{Q_j^*}|f|^p\,\dd v
 \right)^{q/p}.
\]
Apply H\"older's inequality with exponents
\[
 \frac{p}{p-q}
 \quad\text{and}\quad
 \frac{p}{q}.
\]
Then
\[
\begin{aligned}
 \int_\Omega |Rf|^q\delta^q\,\dd\mu
 &\le
 C
 \left(\sum_j c_j^{p/(p-q)}\right)^{(p-q)/p}
 \left(\sum_j\int_{Q_j^*}|f|^p\,\dd v\right)^{q/p}  \\
 &\le
 C
 \left(\sum_j
 \Dcal_{\mu,p,q,\rho}(a_j)^r\right)^{q/r}
 \|f\|_{A^p}^q .
\end{aligned}
\]
This proves sufficiency for $q<p$.

We prove necessity for $p\le q$.  Let $a\in\Omega$ and assume
$\delta(a)<\eta$, where $\eta$ is from \cref{lem:derivative-peaks}.  Let
$F_a$ be the function from that lemma.  The embedding estimate gives
\[
 \int_\Omega |RF_a(z)|^q\delta(z)^q\,\dd\mu(z)
 \le C .
\]
On $\Q_{\epsilon\rho}(a)$,
\[
 |RF_a(z)|^q\delta(z)^q
 \ge c\,\V_\rho(a)^{-q/p}.
\]
Therefore
\[
 \mu(\Q_{\epsilon\rho}(a))
 \le
 C\V_\rho(a)^{q/p}.
\]
If $\delta(a)\ge\eta$, then $a$ lies in a fixed compact subset of
$\Omega$.  For small $\rho$ the polydiscs $\Q_{\epsilon\rho}(a)$ lie in
another fixed compact subset.  Since $\mu$ is locally finite and
$\V_\rho(a)\ge c_\eta>0$ there,
\[
 \frac{\mu(\Q_{\epsilon\rho}(a))}{\V_\rho(a)^{q/p}}
 \le C_\eta .
\]
By \cref{lem:radius-independent-companion},
\[
 \sup_{z\in\Omega}\Dcal_{\mu,p,q,\rho}(z)<\infty .
\]

We prove necessity for $q<p$.  Color the lattice into finitely many
separated sublattices as in \cref{lem:projection-atoms}.  It is enough
to treat one color.  The finitely many points with
$\delta(a_j)\ge\eta$ give a finite contribution.  For the boundary
sublattice let $G_j$ be the projection atoms from that lemma.  Let
$\{r_j(t)\}$ be the Rademacher functions on $[0,1]$.  For a finitely
supported scalar sequence $\{c_j\}$ put
\[
 F_t(z)=\sum_j c_jr_j(t)G_j(z).
\]
The sets $E_{a_j}$ are disjoint, and $B$ is bounded on $L^p$.
Thus, for each $t$,
\[
 \|F_t\|_{A^p}^p\le C\sum_j|c_j|^p.
\]
In particular,
\[
 \int_0^1\|F_t\|_{A^p}^p\,\dd t
 \le C\sum_j|c_j|^p .
\]
The embedding estimate gives, for each $t$,
\[
 \int_\Omega |RF_t(z)|^q\delta(z)^q\,\dd\mu(z)
 \le
 C\|F_t\|_{A^p}^q .
\]
After integration in $t$ and H\"older's inequality on the probability
space $[0,1]$, this gives
\[
\begin{aligned}
 \int_\Omega
 \int_0^1 |RF_t(z)|^q\,\dd t\,\delta(z)^q\,\dd\mu(z)
 &\le
 C
 \int_0^1 \|F_t\|_{A^p}^q\,\dd t\\
 &\le
 C
 \left(\int_0^1 \|F_t\|_{A^p}^p\,\dd t\right)^{q/p}\\
 &\le
 C\left(\sum_j |c_j|^p\right)^{q/p}.
\end{aligned}
\]
By Khinchine's inequality, as in Luecking's embedding method
\cite{Luecking1993},
\[
 \int_0^1 |RF_t(z)|^q\,\dd t
 \asymp
 \left(\sum_j |c_j|^2 |RG_j(z)|^2\right)^{q/2}.
\]
On $Q_j=\Q_{\epsilon\rho}(a_j)$ the $j$-term gives
\[
 \left(\sum_k |c_k|^2 |RF_k(z)|^2\right)^{q/2}\delta(z)^q
 \ge
 c |c_j|^q V_j^{-q/p}.
\]
Hence
\[
 \sum_j |c_j|^q
 \frac{\mu(Q_j)}{V_j^{q/p}}
 \le
 C\left(\sum_j |c_j|^p\right)^{q/p}.
\]
Set
\[
 b_j=\frac{\mu(Q_j)}{V_j^{q/p}}
 =
 \Dcal_{\mu,p,q,\rho}(a_j)^q .
\]
Put $x_j=|c_j|^q$.  Then
\[
 \sum_j x_j b_j
 \le
 C\left(\sum_j x_j^{p/q}\right)^{q/p}.
\]
Since $p/q>1$, duality for $\ell^{p/q}$ gives
\[
 \{b_j\}\in \ell^{p/(p-q)}.
\]
Thus
\[
 \sum_j
 \Dcal_{\mu,p,q,\rho}(a_j)^r<\infty,
 \qquad r=\frac{pq}{p-q}.
\]
Repeating this for the finitely many colors gives the same sum over the
whole lattice.
By the lattice form of the integral condition,
\[
 \Dcal_{\mu,p,q,\rho}\in L^r(d\lambda_\Omega).
\]
We now prove the norm estimates with corrected averages.
The unweighted bounds above need not be sharp.  Fix
$K=\{z:\delta(z)\ge\eta_0\}$, where $\eta_0>0$ is smaller than
the boundary-peak threshold.  Put
$M_K^q=\int_K |z|^q\delta(z)^q\,\dd\mu(z)$.  Cauchy's estimate and
$Rf(0)=0$ show that
\[
 \left(\int_K|Rf|^q\delta^q\,\dd\mu\right)^{1/q}
 \le C_K M_K\|f\|_{A^p}.
\]
Conversely, test the embedding on $f(z)=z_l$ for each $l$.
Their $A^p$ norms are fixed, and $\sum_l|z_l|^q\asymp|z|^q$.
Thus $M_K\le C_K C^{1/q}$.

Outside $K$, $|z|\ge c_K>0$.  Apply the preceding lattice estimates
to the restriction of $\mu$ to this boundary region.  The corrected
averages control its norm.  Conversely, the boundary tests bound
these averages by $C^{1/q}$.

Finitely many lattice cells meet the transition region.  The
coordinate tests control them on a larger compact set.  By bounded
overlap, the compact contribution to the norm of
$\Dcal^{\circ}_{\mu,p,q,\rho}$ is comparable to $M_K$.  This holds
for both $L^\infty$ and $L^r(d\lambda_\Omega)$.  We obtain the two
norm comparisons.  Only finitely many cells meet each fixed compact
set.  Thus corrected and uncorrected averages have the same
membership conditions.
\end{proof}

\subsection{Companion operators between Bergman spaces}

\begin{theorem}\label{thm:p-le-q}
Let $1<p\le q<\infty$.  Let $g\in\Hol(\Omega)$ and let
$\varphi\in\mathcal H(\Omega)$ satisfy $\varphi(0)=0$.  Assume that
$\mu_{g,\varphi,q}$ is locally finite.  Then
$C_{g,\varphi}:A^p(\Omega)\to A^q(\Omega)$ is bounded if and only if
\[
 \Dcal_{g,\varphi,p,q,\rho}\in L^\infty(\Omega).
\]
Moreover,
\[
 \|C_{g,\varphi}\|_{A^p\to A^q}
 \asymp
 \|\Dcal^{\circ}_{g,\varphi,p,q,\rho}\|_{L^\infty}.
\]
The operator $C_{g,\varphi}:A^p(\Omega)\to A^q(\Omega)$ is compact if
and only if
\[
 \Dcal_{g,\varphi,p,q,\rho}(z)\to0,
 \qquad z\to\partial\Omega .
\]
If $C_{g,\varphi}$ is bounded, then
\[
 \|C_{g,\varphi}\|_{e,A^p\to A^q}
 \asymp
 \limsup_{z\to\partial\Omega}\Dcal_{g,\varphi,p,q,\rho}(z).
\]
\end{theorem}

\begin{proof}
Put
\[
 \mu=\mu_{g,\varphi,q}.
\]
By \cref{prop:companion-reduction},
\[
\begin{aligned}
 \|C_{g,\varphi}f\|_{A^q}^q
 &\asymp
 \int_\Omega |Rf(w)|^q\delta(w)^q\,\dd\mu(w).
\end{aligned}
\]
This is first proved for $f\in\Hol(\overline\Omega)$.  By
\cref{lem:radial-approximation}, this class is dense in $A^p(\Omega)$.
Hence boundedness and the norm estimate follow from
\cref{thm:derivative-carleson}.  The operator on $A^p(\Omega)$ is the
unique bounded extension of its action on
$\Hol(\overline\Omega)$.

Assume
\[
 \Dcal_{g,\varphi,p,q,\rho}(z)\to0,
 \qquad z\to\partial\Omega .
\]
Let $\varepsilon>0$.  Choose $K\Subset\Omega$ such that the Carleson
constant of $\mu|_{\Omega\setminus K}$ is at most $\varepsilon^q$, that is,
\[
 \|\Dcal_{\mu\mathbf 1_{\Omega\setminus K},p,q,\rho}\|_\infty
 \le C\varepsilon .
\]
Let $\{f_m\}$ be bounded in $A^p$ and let $f_m\to0$ locally uniformly.
Cauchy's estimate on a fixed neighbourhood of $K$ gives
$Rf_m\to0$ uniformly on $K$.  Since $\mu$ is locally finite, we have
\[
\begin{aligned}
 \|C_{g,\varphi}f_m\|_{A^q}^q
 &\asymp
 \int_K |Rf_m|^q\delta^q\,\dd\mu
 +
 \int_{\Omega\setminus K}|Rf_m|^q\delta^q\,\dd\mu\\
 &\le
 o(1)+C\varepsilon^q\sup_m\|f_m\|_{A^p}^q .
\end{aligned}
\]
Let $m\to\infty$ and then let $\varepsilon\to0$.  We get
\[
 \|C_{g,\varphi}f_m\|_{A^q}\to0 .
\]
Now let $\{u_m\}$ be bounded in $A^p$.  By Montel's theorem, a
subsequence converges locally uniformly to some $u\in A^p$.  Applying
the preceding estimate to $u_{m_k}-u$ gives
\[
 \|C_{g,\varphi}u_{m_k}-C_{g,\varphi}u\|_{A^q}\to0 .
\]
Thus $C_{g,\varphi}$ is compact.

Conversely, assume that $C_{g,\varphi}$ is compact.  Let
$a_m\to\partial\Omega$.  Then $F_{a_m}$ is bounded in $A^p$ and
$F_{a_m}\to0$ locally uniformly.  Hence
\[
 \|C_{g,\varphi}F_{a_m}\|_{A^q}\to0 .
\]
By \cref{prop:companion-reduction} and \cref{lem:derivative-peaks},
\[
\begin{aligned}
 \|C_{g,\varphi}F_{a_m}\|_{A^q}^q
 &\gtrsim
 \int_{\Q_{\epsilon\rho}(a_m)}
 |RF_{a_m}(w)|^q\delta(w)^q\,\dd\mu(w)\\
 &\gtrsim
 \V_\rho(a_m)^{-q/p}\mu(\Q_{\epsilon\rho}(a_m)).
\end{aligned}
\]
Therefore
\[
 \Dcal_{g,\varphi,p,q,\rho}(a_m)\to0 .
\]
The radius lemma gives the boundary limit.

For the essential norm, use radial dilations.  Let
\[
 S_r f(z)=f(rz),
 \qquad 0<r<1 .
\]
Then $S_r:A^p\to A^p$ is compact.  Hence
$C_{g,\varphi}S_r$ is compact.  For $\|f\|_{A^p}\le1$,
\[
\begin{aligned}
 \|(C_{g,\varphi}-C_{g,\varphi}S_r)f\|_{A^q}^q
 &\asymp
 \int_\Omega |R(f-S_rf)(w)|^q\delta(w)^q\,\dd\mu(w)\\
&=
 \int_{\{\delta(w)<\tau\}} |R(f-S_rf)(w)|^q\delta(w)^q\,\dd\mu(w)\\
 &\quad+
 \int_{\{\delta(w)\ge \tau\}} |R(f-S_rf)(w)|^q\delta(w)^q\,\dd\mu(w).
\end{aligned}
\]
For fixed $\tau>0$, the second term tends to $0$ as $r\to1^-$.  The first
term is bounded by
\[
 C\|\Dcal_{\mu\mathbf 1_{\{\delta<\tau\}},p,q,\rho}\|_\infty^q .
\]
Therefore
\[
\|C_{g,\varphi}\|_{e,A^p\to A^q}
\le
 C\lim_{\tau\to0^+}
 \|\Dcal_{\mu\mathbf 1_{\{\delta<\tau\}},p,q,\rho}\|_\infty
 =
 C\limsup_{z\to\partial\Omega}
 \Dcal_{g,\varphi,p,q,\rho}(z).
\]
Let $K:A^p\to A^q$ be compact.  For every boundary sequence $a_m$,
\[
 \|KF_{a_m}\|_{A^q}\to0 .
\]
Thus
\[
\begin{aligned}
 \|C_{g,\varphi}-K\|_{A^p\to A^q}^q
 &\ge
 c\limsup_m \|C_{g,\varphi}F_{a_m}\|_{A^q}^q\\
 &\ge
 c\limsup_m
 \Dcal_{g,\varphi,p,q,\rho}(a_m)^q .
\end{aligned}
\]
Take the supremum over all boundary sequences.  This gives the lower
estimate.
\end{proof}

\begin{theorem}\label{thm:q-less-p}
Let $1<q<p<\infty$.  Let $g\in\Hol(\Omega)$ and let
$\varphi\in\mathcal H(\Omega)$ satisfy $\varphi(0)=0$.  Assume that
$\mu_{g,\varphi,q}$ is locally finite.  Put
\[
 r=\frac{pq}{p-q}.
\]
Then the following conditions are equivalent:
\[
\begin{aligned}
& C_{g,\varphi}:A^p(\Omega)\to A^q(\Omega) \text{ is bounded},\\
& C_{g,\varphi}:A^p(\Omega)\to A^q(\Omega) \text{ is compact},\\
& \Dcal_{g,\varphi,p,q,\rho}\in L^r(\Omega,d\lambda_\Omega).
\end{aligned}
\]
Moreover,
\[
 \|C_{g,\varphi}\|_{A^p\to A^q}
 \asymp
 \|\Dcal^{\circ}_{g,\varphi,p,q,\rho}\|_{L^r(d\lambda_\Omega)} .
\]
\end{theorem}

\begin{proof}
Put
\[
 \mu=\mu_{g,\varphi,q}.
\]
By \cref{prop:companion-reduction} the norm identity holds on
$\Hol(\overline\Omega)$.  By \cref{lem:radial-approximation} this
class is dense in $A^p(\Omega)$.  Therefore boundedness and the norm
estimate follow from \cref{thm:derivative-carleson}, and
$C_{g,\varphi}$ denotes the resulting bounded extension.
It remains to prove that boundedness implies compactness.

Assume
\[
 \Dcal_{g,\varphi,p,q,\rho}\in L^r(d\lambda_\Omega).
\]
Let $K_m=\{z:\delta(z)\ge 1/m\}$ and let
\[
 \dd\mu_m=\mathbf 1_{\Omega\setminus K_m}\,\dd\mu .
\]
By the $q<p$ part of \cref{thm:derivative-carleson},
\[
 \sup_{\|f\|_{A^p}\le1}
 \int_{\Omega\setminus K_m}|Rf|^q\delta^q\,\dd\mu
 \le
 C
 \|\Dcal_{\mu_m,p,q,\rho}\|_{L^r(d\lambda_\Omega)}^q .
\]
The right hand side tends to $0$ by dominated convergence.  On each
$K_m$ the convergence $f_j\to0$ locally uniformly implies
$Rf_j\to0$ uniformly.  Also
\[
 \int_{K_m}\delta^q\,\dd\mu<\infty .
\]
Thus
\[
 \int_{K_m}|Rf_j|^q\delta^q\,\dd\mu
 \le
 \sup_{K_m}|Rf_j|^q
 \int_{K_m}\delta^q\,\dd\mu
 \to0 .
\]
Hence, for every bounded sequence $\{f_j\}$ in $A^p$ which converges to
$0$ locally uniformly,
\[
 \|C_{g,\varphi} f_j\|_{A^q}\to0 .
\]
Let $\{u_j\}$ be bounded in $A^p$.  By Montel's theorem, a subsequence
$u_{j_k}$ converges locally uniformly to some $u\in A^p$.  Applying the
last estimate to $u_{j_k}-u$ gives
\[
 \|C_{g,\varphi}u_{j_k}-C_{g,\varphi}u\|_{A^q}\to0 .
\]
This is compactness.
\end{proof}

\begin{unnumberedcorollary}
Let $1<p,q<\infty$ and let $C_g=C_{g,\operatorname{id}}$.
\[
\begin{array}{ll}
\textup{(i)} & p<q:\quad C_g:A^p\to A^q
\text{ is bounded if and only if }g=0,\\[2mm]
\textup{(ii)}& p=q:\quad C_g:A^p\to A^p
\text{ is bounded if and only if }g\in H^\infty(\Omega),\\[2mm]
\textup{(iii)}&p=q:\quad C_g:A^p\to A^p
\text{ is compact if and only if }g=0,\\[2mm]
\textup{(iv)}&q<p:\quad g\in H^\infty(\Omega)
\text{ implies }C_g:A^p\to A^q\text{ is compact.}
\end{array}
\]
\end{unnumberedcorollary}

\begin{proof}
If $p<q$, boundedness gives
\[
 \Acal_{g,q,\rho}(z)
 \le
 C\V_\rho(z)^{1/p-1/q}.
\]
The exponent is positive.  Hence
\[
 \Acal_{g,q,\rho}(z)\to0
 \quad z\to\partial\Omega .
\]
The sub-mean property gives
\[
 \sup_{\delta(z)<\varepsilon}|g(z)|\to0 .
\]
The maximum principle on the exhaustion
$\{z:\delta(z)>\varepsilon\}$ gives $g=0$.

If $p=q$, then
\[
 \Dcal_{g,p,p,\rho}=\Acal_{g,p,\rho}.
\]
Boundedness of $\Acal_{g,p,\rho}$ is equivalent to
$g\in H^\infty(\Omega)$ by the sub-mean inequality.  Compactness is
equivalent to boundary vanishing of $\Acal_{g,p,\rho}$, and the previous
maximum principle argument gives $g=0$.

If $q<p$ and $g\in H^\infty$, then
\[
 \Dcal_{g,p,q,\rho}(z)^r
 \le
 C\|g\|_\infty^r\V_\rho(z),
 \qquad r=\frac{pq}{p-q}.
\]
Therefore
\[
 \int_\Omega \Dcal_{g,p,q,\rho}(z)^r\,\dd\lambda_\Omega(z)
 \le
 C\|g\|_\infty^r\int_\Omega \dd v(z)<\infty .
\]
The compactness follows from \cref{thm:q-less-p}.
\end{proof}

\begin{corollary}\label{cor:disc-model}
Let $\Omega=\mathbb D$ and let $C_g=C_{g,\operatorname{id}}$.  Then
\[
 C_g f(z)=I_gf(z)=\int_0^z f'(\zeta)g(\zeta)\,\dd\zeta .
\]
Let $D(a,\rho)$ be a fixed Bergman disc.  Then
\[
 \V_\rho(a)\asymp (1-|a|^2)^2
\]
and
\[
 |g(a)|\lesssim\Acal_{g,q,\rho}(a).
\]
Hence
\[
 \Dcal_{g,p,q,\rho}(a)
 \asymp
 (1-|a|^2)^{2(1/q-1/p)}
 \left(\avg_{D(a,\rho)}|g(w)|^q\,\dd A(w)\right)^{1/q}.
\]
If $q<p$ and $r=pq/(p-q)$, then
\[
 C_g:A^p(\mathbb D)\to A^q(\mathbb D)
 \text{ is bounded}
\]
if and only if
\[
 g\in A^r(\mathbb D).
\]
The same condition is equivalent to compactness.
\end{corollary}

\begin{proof}
The identity follows from $\zeta=tz$:
\[
 C_g f(z)
 =
 \int_0^1 tz f'(tz)g(tz)\,\frac{\dd t}{t}
 =
 \int_0^z f'(\zeta)g(\zeta)\,\dd\zeta .
\]
For Bergman discs,
\[
 v(D(a,\rho))\asymp (1-|a|^2)^2 .
\]
Here $dA$ denotes planar area measure.  Also
\[
 \dd\lambda(a)\asymp \frac{\dd A(a)}{(1-|a|^2)^2}.
\]
The sub-mean estimate gives $|g(a)|\lesssim\Acal_{g,q,\rho}(a)$.
There can be no reverse pointwise comparison at a zero of $g$.
The stated formula for $\Dcal_{g,p,q,\rho}$ follows from its
definition.

If $q<p$, then
\[
\begin{aligned}
 \int_{\mathbb D}
 \Dcal_{g,p,q,\rho}(a)^r\,\dd\lambda(a)
 &\asymp
 \int_{\mathbb D}\Acal_{g,q,\rho}(a)^r\,\dd A(a)\\
 &\asymp\int_{\mathbb D}|g(a)|^r\,\dd A(a),
\end{aligned}
\]
because
\[
 2r(1/q-1/p)=2.
\]
For the last comparison use the sub-mean estimate in one direction.
In the other, $r>q$ and Jensen's inequality bound the local $L^q$
average raised to $r/q$ by the local $L^r$ average; integration and
bounded overlap finish the argument.
The conclusion follows from \cref{thm:q-less-p}.
\end{proof}

\subsection{Examples and sharpness}\label{subsec:examples-sharpness}

The pull-back map is essential.  Let $\Omega=\mathbb D$,
\[
 \varphi_a(z)=az,\qquad 0<|a|<1,
\]
and let $g=1$.  Then $\mu_{1,\varphi_a,q}$ is supported in
\[
 \{w: |w|\le |a|\}.
\]
Hence
\[
 \Dcal_{1,\varphi_a,p,q,\rho}(z)\to0,
 \qquad z\to\partial\mathbb D .
\]
Therefore
\[
 C_{1,\varphi_a}:A^p(\mathbb D)\to A^q(\mathbb D)
\]
is compact for all $1<p,q<\infty$.  This does not happen for
$\varphi=\operatorname{id}$ when $p<q$, unless $g=0$.

The same example also separates the Schatten theory from the diagonal
case.  For $q=2$ the measure $\mu_{1,\varphi_a,2}$ has compact support
in $\mathbb D$.  Hence
\[
 \Acal_{1,\varphi_a,2,\rho}
 \in L^s(\mathbb D,d\lambda)
 \qquad 0<s<\infty .
\]
By \cref{cor:strict-self-map},
\[
 C_{1,\varphi_a}\in \Schatten_s(A^2(\mathbb D)),
 \qquad 0<s<\infty .
\]
By contrast, \cref{cor:no-schatten} shows that
$C_{1,\operatorname{id}}$ belongs to no Schatten class.  Thus the
composition part is essential to the result.

The same shrinking maps expose why the weight at the origin is needed
in quantitative estimates.  If $a=\varepsilon>0$ and $g=1$, then
$C_{1,\varphi_\varepsilon}f=f(\varepsilon z)-f(0)$.  With normalized
monomials in $A^2(\mathbb D)$,
\[
 \|C_{1,\varphi_\varepsilon}\|=\varepsilon,
 \qquad
 \|C_{1,\varphi_\varepsilon}\|_{\Schatten_s}^s
 =\frac{\varepsilon^s}{1-\varepsilon^s}\quad(s>0).
\]
For a fixed polydisc around the origin the unweighted pull-back
average, and its $L^{s/2}$ integral, stay bounded below as
$\varepsilon\downarrow0$.  The averages formed from
$|w|^q\mu_{1,\varphi_\varepsilon,q}$ instead have the correct decay.

\begin{unnumberedexample}[A finite-type model in complex dimension $n$]
Let $n,m\ge2$ be integers, write $(z,w)\in\mathbb C^{n-1}\times
\mathbb C$, and set
\[
 \Omega_{m,n}=\{(z,w): |z|^2+|w|^{2m}<1\}.
\]
This is a smooth convex domain of type $2m$ at $(a,0)$ when
$|a|=1$.  Near this locus, the distinguished radii have orders
$\delta$ in the complex normal direction, $\delta^{1/2}$ in the
$n-2$ remaining $z$ directions, and $\delta^{1/(2m)}$ in the $w$
direction.  Thus $\V_\rho\asymp\delta^{n+1/m}$.  Put
$\varphi(z,w)=(z,0)$ and $g(z,w)=w^k$ for an integer $k\ge0$.
The pull-back measure is carried by $\{w=0\}$.  For a polydisc
centered at $(a,0)$ with $|a|\to1$,
\[
 \mu_{g,\varphi,q}(\Q_\rho((a,0)))
 \asymp (1-|a|)^{n+1/m+kq/(2m)},\qquad
 \Acal_{g,\varphi,q,\rho}((a,0))
 \asymp(1-|a|)^{k/(2m)}.
\]
To obtain the measure estimate, integrate $|w|^{kq}$ over the
fibers $|w|^{2m}<1-|z|^2$.  After rescaling, the depth ratio in
the pull-back measure has bounded average.  The same estimate holds
for polydiscs meeting $\{w=0\}$.  Hence, when $p=q$,
$C_{g,\varphi}$ is bounded for every $k$ and compact exactly when
$k>0$.

The Schatten threshold can be computed for every $s>0$ without a
Toeplitz trace theorem.  For a multi-index
$\beta\in\mathbb N_0^{n-1}$, write $|\beta|=\sum_j\beta_j$ and
$\beta!=\prod_j\beta_j!$.  If $N=|\beta|\ge1$, orthogonality of
monomials gives
\[
 C_{w^k,\varphi}(z^\beta)=\frac{N}{N+k}z^\beta w^k,
 \qquad C_{w^k,\varphi}(z^\beta w^b)=0\quad (b\ge1).
\]
Integrating first in $w$ and then in the unit ball of $\mathbb C^{n-1}$
gives
\[
 \|z^\beta w^k\|_{A^2(\Omega_{m,n})}^2
 =\frac{\pi^n\beta!}{k+1}
 \frac{\Gamma(1+(k+1)/m)}{\Gamma(N+n+(k+1)/m)}.
\]
It follows that the nonzero singular values, indexed by $\beta$, are
\[
 s_\beta(C_{w^k,\varphi})
 =\frac{N}{N+k}\left[
 \frac{\Gamma(1+(k+1)/m)\Gamma(N+n+1/m)}
 {(k+1)\Gamma(1+1/m)\Gamma(N+n+(k+1)/m)}
 \right]^{1/2}
 \asymp N^{-k/(2m)}.
\]
There are $\binom{N+n-2}{n-2}\asymp N^{n-2}$ multi-indices of
length $N$.  Consequently, the ordered singular values have order
$j^{-k/[2m(n-1)]}$.  Hence
\[
 \sum_\beta s_\beta(C_{w^k,\varphi})^s
 \asymp\sum_{N\ge1}N^{n-2-sk/(2m)}<\infty
 \quad\Longleftrightarrow\quad sk>2m(n-1).
\]
In particular, when $n=2$ the threshold is $sk>2m$.  The type
controls the decay of each singular value.  The dimension controls
their multiplicity.  This calculation also covers $s<2$, without
asserting the general derivative Toeplitz criterion in that range.
\end{unnumberedexample}

The exponent in the case $q<p$ is sharp.  Let
\[
 r=\frac{pq}{p-q}
\]
and put
\[
 g_\alpha(z)=(1-z)^{-\alpha},\qquad \alpha>0 .
\]
Then
\[
 g_\alpha\in A^r(\mathbb D)
 \quad\text{if and only if}\quad
 \alpha r<2 .
\]
By \cref{cor:disc-model},
\[
 C_{g_\alpha}:A^p(\mathbb D)\to A^q(\mathbb D)
\]
is bounded, equivalently compact, if and only if
\[
 \alpha<\frac2r .
\]
Thus unbounded symbols are allowed when $q<p$.  The condition
\[
 \Dcal_{g,p,q,\rho}\in L^r(d\lambda)
\]
cannot be replaced by $g\in H^\infty$.

\clearpage
\section{Schatten Classes and Compact-Range Singular Values}

The derivative Toeplitz form separates the trace-ideal calculation
from the radial primitive.  We use it below for companion operators.

\subsection{Derivative Toeplitz forms}

In this section $C_{g,\varphi}$ acts on $A^2(\Omega)$.  For a positive Borel
measure $\mu$ define the derivative Toeplitz form
\[
 \mathfrak s_\mu(f,h)
 =
 \int_\Omega Rf(z)\overline{Rh(z)}\delta(z)^2\,\dd\mu(z).
\]
When the form is bounded on $A^2(\Omega)$, let $S_\mu$ be the positive
operator such that
\[
 \langle S_\mu f,h\rangle_{A^2}
 =
 \mathfrak s_\mu(f,h).
\]

We use one Schatten parameter for $C_{g,\varphi}$ and another one
for $S_\mu$.  If
\[
 C_{g,\varphi}\in\Schatten_s,
\]
then the positive comparison is made for
\[
 C_{g,\varphi}^*C_{g,\varphi}
 \in\Schatten_{s/2}.
\]
The argument below uses $s/2\ge1$.  It therefore gives the range
$s\ge2$ for $C_{g,\varphi}$.

\begin{theorem}\label{thm:derivative-toeplitz}
Let $\mu$ be a positive locally finite Borel measure on $\Omega$ with
$\mu(\{0\})=0$ and let $1\le t<\infty$.  Write
$\dd\nu(w)=|w|^2\,\dd\mu(w)$.  Then
\[
 S_\mu\in\Schatten_t(A^2(\Omega))
 \quad\Longleftrightarrow\quad
 \widehat\mu_\rho(z)=\frac{\mu(\Q_\rho(z))}{\V_\rho(z)}
 \in L^t(\Omega,d\lambda_\Omega).
\]
The quantitative estimate is
\begin{equation}\label{eq:corrected-derivative-trace}
 \|S_\mu\|_{\Schatten_t}^t\asymp
 \int_\Omega
 \left(\frac{\nu(\Q_\rho(z))}{\V_\rho(z)}\right)^t
 \dd\lambda_\Omega(z).
\end{equation}
The constants in this comparison are independent of $\mu$.
\end{theorem}

\begin{proof}
Let $H_z$ represent the functional $f\mapsto\delta(z)Rf(z)$.
Put
\[
 h(z)=\|H_z\|^2
 =\delta(z)^2R_z\overline{R_z}K(z,z).
\]
On a fixed compact set $K_0\Subset\Omega$, Cauchy's estimates give
an upper bound for $h$.  The coordinate functions give a lower
bound.  Together they yield
\begin{equation}\label{eq:interior-derivative-kernel}
 h(z)\asymp_{K_0}|z|^2,\qquad z\in K_0.
\end{equation}
Choose a thin boundary collar $B$.  The local derivative Cauchy
estimate gives $h(z)\lesssim\V_\rho(z)^{-1}$ there.
For the reverse bound, evaluate $\delta(z)RF_z(z)$ using the peak
from \cref{lem:derivative-peaks}.  The reproducing norm gives
\begin{equation}\label{eq:boundary-derivative-kernel}
 h(z)\asymp\V_\rho(z)^{-1},\qquad z\in B.
\end{equation}
This estimate holds on $B$.  It fails at $z=0$.

First restrict $\mu$ to $B$ and write $\mu_B=\mu|_B$.
Take a sufficiently fine McNeal lattice $\{a_j\}$ in $B$.
Choose a Borel partition $B=\bigsqcup_j E_j$ with
$E_j\subset Q_j=\Q_\rho(a_j)$.  The polydiscs $Q_j$ have
bounded overlap.
Put $d_j=\mu_B(E_j)/V_j$ and $S_j=S_{\mu_B|_{E_j}}$.
Then $S_{\mu_B}=\sum_jS_j$ as a monotone sum.
The rank-one integral formula and
\eqref{eq:boundary-derivative-kernel} give
\[
 \|S_j\|_{\Schatten_1}=\Tr S_j
 =\int_{E_j}h(z)\,\dd\mu(z)\lesssim d_j.
\]
For any finitely supported scalars $c_j$, the derivative Carleson
upper bound in \cref{thm:derivative-carleson} (with $p=q=2$) gives
\[
 \left\|\sum_j c_j\frac{S_j}{d_j}\right\|_{\mathcal B(A^2)}
 \lesssim\|c\|_{\ell^\infty};
\]
Omit terms with $d_j=0$.  The measure
$\sum_j |c_j|\mu_B|_{E_j}/d_j$ has bounded local averages.
This follows from volume comparability and bounded overlap.
The trace estimate bounds the same map from $\ell^1$ to
$\Schatten_1$.  For a finite set $F$, write

\[
 T_F(c)=\sum_{j\in F,\,d_j>0}c_j\frac{S_j}{d_j}.
\]
Both bounds are independent of $F$.  Interpolating
$T_F:\ell^\infty\to\mathcal B(A^2)$ and
$T_F:\ell^1\to\Schatten_1$, and then taking $c_j=d_j$, gives
$\|\sum_{j\in F}S_j\|_{\Schatten_t}
\le C_t(\sum_{j\in F}d_j^t)^{1/t}$.  Thus, for $1<t<\infty$,
\begin{equation}\label{eq:trace-upper-interpolation}
 \|S_{\mu_B}\|_{\Schatten_t}^t\lesssim\sum_jd_j^t.
\end{equation}
The case $t=1$ is the trace estimate itself.  Monotone limits remove
all finite-support restrictions.

For the reverse estimate, use the polydiscs $Q_j$.
Put $u_j=V_j^{1/2}H_{a_j}$.  The local derivative Cauchy estimate
and bounded overlap show that $\{u_j\}$ is a Bessel family.
For $f\in A^2$,
\[
 \sum_j|\langle f,u_j\rangle|^2
 =\sum_jV_j\delta(a_j)^2|Rf(a_j)|^2
 \lesssim\|f\|_{A^2}^2.
\]
At $a_j$, \eqref{eq:boundary-derivative-kernel} gives
$|\delta(a_j)Ru_j(a_j)|\gtrsim V_j^{-1/2}$.
First fix a reference radius $\rho_0$ larger than $\rho$.
At this radius, the diagonal and evaluation bounds give
\[
 h(a_j)\asymp\V_{\rho_0}(a_j)^{-1},
 \qquad
 \sup_{z\in\Q_{\rho_0/2}(a_j)}
 \delta(a_j)|RH_{a_j}(z)|
 \le C\V_{\rho_0}(a_j)^{-1}.
\]
Apply Cauchy's estimate to $RH_{a_j}$ on
$\Q_{\rho_0/2}(a_j)$.  For some $\gamma>0$,
\[
 \sup_{z\in\Q_\rho(a_j)}
 \delta(a_j)|RH_{a_j}(z)-RH_{a_j}(a_j)|
 \le C\left(\frac{\rho}{\rho_0}\right)^\gamma h(a_j).
\]
Choose $\rho$ so that the right side is at most $h(a_j)/2$.
Since $\delta(a_j)RH_{a_j}(a_j)=h(a_j)$, it follows that
$\delta(a_j)|RH_{a_j}(z)|\ge h(a_j)/2$ on $Q_j$.
Now use $\delta(z)\asymp\delta(a_j)$ there.  Multiplying by
$V_j^{1/2}$ gives
\[
 \inf_{z\in Q_j}\delta(z)|Ru_j(z)|
 \ge c V_j^{1/2}h(a_j)
 \ge c_\rho V_j^{-1/2}.
\]
The constants are uniform in the boundary centers.  Consequently
\[
 \langle S_{\mu_B}u_j,u_j\rangle
 \gtrsim\frac{\mu_B(Q_j)}{V_j}.
\]
Let $T:\ell^2\to A^2$ satisfy $Te_j=u_j$.  It is bounded by the
Bessel estimate.  For $t\ge1$, apply Jensen's inequality to
$T^*S_{\mu_B}T$.  The ideal property then gives
\[
 \sum_j\langle S_{\mu_B}u_j,u_j\rangle^t
 \le\Tr(T^*S_{\mu_B}T)^t
 \le\|T\|^{2t}\Tr(S_{\mu_B}^t).
\]
This proves the converse to \eqref{eq:trace-upper-interpolation},
with $d_j$ replaced by $\mu_B(Q_j)/V_j$.  Bounded overlap gives the
same $\ell^t$ norm for cell and polydisc masses.

The complementary measure $\mu_0=\mu|_{\Omega\setminus B}$ is
supported on a compact core.  Equation
\eqref{eq:interior-derivative-kernel} yields
\[
 \Tr S_{\mu_0}\asymp
 M_0:=\int_{\Omega\setminus B}|z|^2\,\dd\mu(z).
\]
Conversely, testing on the fixed coordinate functions gives
$\|S_{\mu_0}\|\gtrsim M_0$.  Therefore, for every $t\ge1$,
$\|S_{\mu_0}\|_{\Schatten_t}^t\asymp M_0^t$.  Positivity, the triangle
inequality, and monotonicity of eigenvalues show
\[
 \|S_\mu\|_{\Schatten_t}^t\asymp
 \|S_{\mu_0}\|_{\Schatten_t}^t+
 \|S_{\mu_B}\|_{\Schatten_t}^t.
\]
Enlarge the compact set to include the fixed transition band.
The lattice-integral comparison identifies the preceding sum with
the right side of \eqref{eq:corrected-derivative-trace}.
On $B$, $\nu\asymp\mu$.  On the compact set, both measures are
locally finite.  Thus the weighted and unweighted averages have
the same $L^t$ membership condition.
\end{proof}

The upper estimate has a useful extension to exponents below one.
It does not use the diagonal lower estimate in the preceding proof.

\begin{proposition}\label{prop:small-derivative-sufficient}
Let $\mu$ be a positive locally finite Borel measure on $\Omega$ with
$\mu(\{0\})=0$, and let $0<t<1$.  Put
$\dd\nu(w)=|w|^2\,\dd\mu(w)$.  If
\[
 I_t(\mu):=\int_\Omega
 \left(\frac{\nu(\Q_\rho(z))}{\V_\rho(z)}\right)^t
 \dd\lambda_\Omega(z)<\infty,
\]
then $S_\mu\in\Schatten_t(A^2(\Omega))$ and
\[
 \|S_\mu\|_{\Schatten_t}^t\le C_t I_t(\mu).
\]
The converse is not asserted here.
\end{proposition}

\begin{proof}
Choose a thin boundary collar $B$.  Then $|z|\ge c_B>0$ on a
fixed enlargement of $B$.  Take a sufficiently fine McNeal lattice
$\{a_j\}$.  Partition $B$ into disjoint Borel cells
$E_j\subset\Q_\rho(a_j)$.  Put
\[
 V_j=\V_\rho(a_j),\qquad
 d_j=\frac{\mu(E_j)}{V_j},\qquad
 S_j=S_{\mu|_{E_j}}.
\]
We first prove the local estimate
\begin{equation}\label{eq:small-local-trace}
 \Tr(S_j^t)\le C_t d_j^t.
\end{equation}
Define $J_j:A^2(\Omega)\to L^2(E_j,\mu)$ by
$J_jf(z)=\delta(z)Rf(z)$.  Then $S_j=J_j^*J_j$.
Cauchy's estimate on an enlarged polydisc gives
\[
 \sup_{z\in\Q_{c\rho}(a_j)}\delta(z)|Rf(z)|
 \le C V_j^{-1/2}\|f\|_{A^2}
\]
for a fixed $c>1$.  In the distinguished coordinates at $a_j$,
the radii of $\Q_\rho(a_j)$ are smaller than those of the enlarged
polydisc by a uniform factor.  Let $P_{j,N}f$ be the Taylor
polynomial of $Rf$ of total degree at most $N$.  The multivariable
Taylor estimate gives
\[
 \sup_{z\in E_j}\delta(z)|Rf(z)-P_{j,N}f(z)|
 \le C\theta^N V_j^{-1/2}\|f\|_{A^2},
 \qquad 0<\theta<1.
\]
Here the polynomial factor in the Taylor remainder is absorbed by
choosing a slightly larger $\theta<1$.
The operator $f\mapsto\delta P_{j,N}f|_{E_j}$ has rank at most
$\binom{N+n}{n}$.  It follows that
\[
 a_{\binom{N+n}{n}+1}(J_j)
 \le C d_j^{1/2}\theta^N,
 \qquad \|J_j\|\le C d_j^{1/2}.
\]
Since $\binom{N+n}{n}=O((N+1)^n)$, we have
$s_\ell(J_j)\le C d_j^{1/2}\exp(-c\ell^{1/n})$.  Therefore
\[
 \Tr(S_j^t)=\sum_{\ell\ge1}s_\ell(J_j)^{2t}
 \le C_t d_j^t,
\]
which proves \eqref{eq:small-local-trace}.  These constants are
independent of $j$ and of $\mu$.

For finite sums of positive operators, trace subadditivity for
$0<t<1$ gives
\[
 \Tr\left(\sum_{j=1}^M S_j\right)^t
 \le\sum_{j=1}^M\Tr(S_j^t)
 \le C_t\sum_{j=1}^M d_j^t.
\]
Since $|z|\asymp1$ on the collar, bounded overlap and the lattice
comparison show
\[
 \sum_jd_j^t
 \le C\sum_j
 \left(\frac{\nu(\Q_\rho(a_j))}{V_j}\right)^t
 \le C I_t(\mu).
\]
The partial sums are Cauchy in $\Schatten_t$.  Their limit is
$S_{\mu|_B}$ in the quadratic-form sense.

Now consider $K=\Omega\setminus B$.  Put
$M_K=\int_K|z|^2\,\dd\mu(z)$.  Define
$J_Kf=\delta Rf|_K$ as a map into $L^2(K,\mu)$.
Choose $K\Subset r_0\Omega\Subset r_1\Omega\Subset\Omega$.
Cauchy's estimate on $r_1\Omega$ and polynomial approximation on
$r_0\Omega$ give polynomials $P_{j,N}f$ of degree at most $N$.
They satisfy
\[
 \sup_K|\partial_jf-P_{j,N}f|
 \le C\theta^N\|f\|_{A^2},\qquad 1\le j\le n.
\]
Since $Rf(z)=\sum_jz_j\partial_jf(z)$, the map
$f\mapsto\delta\sum_jz_jP_{j,N}f|_K$ approximates $J_K$.
Its error is at most $C M_K^{1/2}\theta^N$, and its rank is at most
$n\binom{N+n}{n}$.  Summing the resulting singular-value bound gives
\[
 \|S_{\mu|_K}\|_{\Schatten_t}^t
 =\sum_{\ell\ge1}s_\ell(J_K)^{2t}
 \le C_t M_K^t.
\]
A finite number of lattice cells cover $K$.  Their volumes are bounded
below, so the lattice comparison gives $M_K^t\le C_t I_t(\mu)$.
Trace subadditivity applied to $S_\mu=S_{\mu|_B}+S_{\mu|_K}$
completes the proof.
\end{proof}

\subsection{Schatten criteria for companions}

\begin{lemma}\label{lem:quadratic-comparison-companion}
Let $g\in\Hol(\Omega)$ and let
$\varphi\in\mathcal H(\Omega)$ satisfy $\varphi(0)=0$.  Put
\[
 \mu=\mu_{g,\varphi,2}.
\]
Assume that $\mu$ is locally finite.
Then $C_{g,\varphi}$ is bounded on $A^2(\Omega)$ if and only if
$S_\mu$ is bounded.  In that case
\[
 C_{g,\varphi}^*C_{g,\varphi}
 \asymp
 S_\mu
\]
in the sense of quadratic forms.
\end{lemma}

\begin{proof}
On the dense subspace $\Hol(\overline\Omega)$,
\cref{prop:companion-reduction} with $q=2$ gives
\[
 \|C_{g,\varphi}f\|_{A^2}^2
 \asymp
 \int_\Omega|Rf(w)|^2\delta(w)^2\,\dd\mu(w).
\]
The constants are independent of $f$.  If $S_\mu$ is bounded,
the right side is bounded by $C\|f\|_{A^2}^2$.  Hence
$C_{g,\varphi}$ extends uniquely to a bounded operator on $A^2$.
Conversely, if $C_{g,\varphi}$ is bounded, the form on the right
extends by density to a bounded positive form.  Its representing
operator is $S_\mu$: local finiteness and locally uniform convergence
of the derivatives identify the extended form with the integral.
Applying the same comparison to the bounded
extensions gives
\[
 \langle C_{g,\varphi}^*C_{g,\varphi}f,f\rangle
 \asymp\langle S_\mu f,f\rangle,
 \qquad f\in A^2(\Omega).
\]
\end{proof}

\begin{lemma}\label{lem:positive-comparison}
Let $A$ and $B$ be positive bounded operators on a Hilbert space.  Assume
that
\[
 c\langle Bx,x\rangle
 \le
 \langle Ax,x\rangle
 \le
 C\langle Bx,x\rangle
 \qquad x\in H .
\]
Then, for every $0<t<\infty$,
\[
 A\in\Schatten_t
 \Longleftrightarrow
 B\in\Schatten_t,
\]
and
\[
 \|A\|_{\Schatten_t}^t\asymp \|B\|_{\Schatten_t}^t .
\]
\end{lemma}

\begin{proof}
Assume first that $B$ is compact.  Then
\[
 0\le A\le C B
\]
implies that $A$ is compact.  The min-max principle for positive compact
operators gives
\[
 \lambda_j(A)\le C\lambda_j(B),
 \qquad j\ge1.
\]
The other form inequality gives
\[
 \lambda_j(B)\le c^{-1}\lambda_j(A),
 \qquad j\ge1.
\]
The same argument starts from the compactness of $A$.  For positive
operators the singular values are the eigenvalues.  Raising to the
power $t$ and summing gives the result for every $0<t<\infty$.
\end{proof}

\begin{theorem}\label{thm:schatten-companion}
Let $g\in\Hol(\Omega)$ and
$\varphi\in\mathcal H(\Omega)$ satisfy $\varphi(0)=0$.  Let
$2\le s<\infty$, and assume that $\mu_{g,\varphi,2}$ is locally
finite.  Then
\[
 C_{g,\varphi}\in\Schatten_s(A^2(\Omega))
 \quad\Longleftrightarrow\quad
 \Acal_{g,\varphi,2,\rho}\in L^s(\Omega,d\lambda_\Omega).
\]
Moreover,
\[
 \|C_{g,\varphi}\|_{\Schatten_s}^s
 \asymp\int_\Omega
 \Acal^{\circ}_{g,\varphi,2,\rho}(z)^s\,\dd\lambda_\Omega(z).
\]
\end{theorem}

\begin{proof}
Put $t=s/2\ge1$ and $\mu=\mu_{g,\varphi,2}$, which satisfies
$\mu(\{0\})=0$.  For a bounded operator $T$ on a Hilbert space,
\[
 T\in\Schatten_s\ \Longleftrightarrow\ T^*T\in\Schatten_t,
 \qquad
 \|T\|_{\Schatten_s}^s=\|T^*T\|_{\Schatten_t}^{t}.
\]
The quadratic comparison in
\cref{lem:quadratic-comparison-companion,lem:positive-comparison}
and the derivative trace criterion in
\cref{thm:derivative-toeplitz} give
\[
\begin{aligned}
 C_{g,\varphi}\in\Schatten_s
 &\Longleftrightarrow S_\mu\in\Schatten_t\\
 &\Longleftrightarrow
 \widehat\mu_\rho:=\frac{\mu(\Q_\rho(\cdot))}{\V_\rho(\cdot)}
 \in L^t(d\lambda_\Omega)\\
 &\Longleftrightarrow
 \Acal_{g,\varphi,2,\rho}\in L^s(d\lambda_\Omega),
\end{aligned}
\]
since $\Acal_{g,\varphi,2,\rho}^2=\widehat\mu_\rho$.
The same comparisons and \eqref{eq:corrected-derivative-trace}
prove the norm formula; its quantitative average uses
$|w|^2\,d\mu(w)$ and hence $\Acal^{\circ}_{g,\varphi,2,\rho}$.
\end{proof}

\begin{corollary}\label{cor:small-schatten-sufficient}
Let $0<s<2$, and suppose that $\mu_{g,\varphi,2}$ is locally finite.
If
\[
 \Acal^{\circ}_{g,\varphi,2,\rho}
 \in L^s(\Omega,d\lambda_\Omega),
\]
then $C_{g,\varphi}\in\Schatten_s(A^2(\Omega))$, with
\[
 \|C_{g,\varphi}\|_{\Schatten_s}^s
 \le C_s\int_\Omega
 \Acal^{\circ}_{g,\varphi,2,\rho}(z)^s
 \dd\lambda_\Omega(z).
\]
In particular, the uncorrected condition
$\Acal_{g,\varphi,2,\rho}\in L^s(d\lambda_\Omega)$ is sufficient.
\end{corollary}

\begin{proof}
Put $t=s/2$ and $\mu=\mu_{g,\varphi,2}$.  By
\cref{prop:small-derivative-sufficient}, $S_\mu\in\Schatten_t$.
The quadratic comparison in \cref{lem:quadratic-comparison-companion}
and monotonicity of eigenvalues give
\[
 \|C_{g,\varphi}\|_{\Schatten_s}^s
 =\Tr(C_{g,\varphi}^*C_{g,\varphi})^t
 \le C_t\Tr(S_\mu^t)
 \le C_t\int_\Omega
 \Acal^{\circ}_{g,\varphi,2,\rho}(z)^s
 \dd\lambda_\Omega(z).
\]
Since $|z|$ is bounded on $\Omega$, the uncorrected condition implies
the corrected one.
\end{proof}

For $t<1$, the diagonal Jensen inequality used in the proof of
\cref{thm:derivative-toeplitz} has the opposite direction.  Thus the
preceding sufficient condition does not prove necessity in this range.

The Hilbert--Schmidt case admits a useful diagonal-kernel formulation.
It also makes explicit the relation with composition--differentiation
operators on the ball studied in \cite{Abkar2024}.

\begin{unnumberedcorollary}
Let $K$ be the Bergman kernel of $\Omega$.  Under the hypotheses of
\cref{thm:schatten-companion}, the following are equivalent:
\[
 C_{g,\varphi}\in\Schatten_2(A^2(\Omega))
\]
and
\[
 \int_{\{\varphi(z)\ne0\}} |g(z)|^2\delta(z)^2
 \left[R_w\overline{R_w}K(w,w)\right]_{w=\varphi(z)}\,\dd v(z)
 <\infty.
\]
Moreover, the last integral is comparable to
$\|C_{g,\varphi}\|_{\Schatten_2}^2$.
\end{unnumberedcorollary}

\begin{proof}
Put $\mu=\mu_{g,\varphi,2}$ and write
\[
 I_\mu=\int_\Omega\delta(w)^2
 R_w\overline{R_w}K(w,w)\,\dd\mu(w).
\]
If $I_\mu<\infty$, the reproducing formula and Cauchy--Schwarz give
$\int|Rf|^2\delta^2\,\dd\mu\le I_\mu\|f\|_{A^2}^2$.
Thus $S_\mu$ is bounded, and so is $C_{g,\varphi}$ by
\cref{lem:quadratic-comparison-companion}.  The Tonelli identity
below gives $\Tr S_\mu=I_\mu<\infty$.  Conversely, if
$C_{g,\varphi}$ is Hilbert--Schmidt,
\cref{lem:quadratic-comparison-companion,lem:positive-comparison}
imply that $S_\mu$ is trace class and hence $I_\mu<\infty$.  In either case,
the quadratic and positive-operator comparisons give
\[
 \|C_{g,\varphi}\|_{\Schatten_2}^2
 =\Tr(C_{g,\varphi}^*C_{g,\varphi})
 \asymp \Tr S_{\mu_{g,\varphi,2}}.
\]
If $\{e_k\}$ is an orthonormal basis of $A^2(\Omega)$, Tonelli's
theorem and differentiation of the reproducing identity give
\[
\begin{aligned}
 \Tr S_{\mu_{g,\varphi,2}}
 &=\int_\Omega \delta(w)^2
   \sum_k|Re_k(w)|^2\,\dd\mu_{g,\varphi,2}(w)\\
 &=\int_\Omega \delta(w)^2
   R_w\overline{R_w}K(w,w)\,\dd\mu_{g,\varphi,2}(w).
\end{aligned}
\]
The definition of the pull-back measure converts the last expression
into the integral in the statement.
\end{proof}

\begin{corollary}\label{cor:alpha-family}
Let $\alpha>0$ and define
\[
 C^{(\alpha)}_{g,\varphi}f(z)
 =\int_0^1 Rf(\varphi(tz))g(tz)t^{\alpha-1}\,\dd t.
\]
Under the respective hypotheses of
\cref{thm:p-le-q,thm:q-less-p,thm:schatten-companion}, their
boundedness, compactness, essential-norm, and Schatten criteria
also hold with $C_{g,\varphi}$ replaced by
$C^{(\alpha)}_{g,\varphi}$.  The same corrected norm comparisons
hold, with constants that may also depend on $\alpha$.
The small-exponent sufficient condition in
\cref{cor:small-schatten-sufficient} also holds for this family.
\end{corollary}

\begin{proof}
The only change is the radial shift.  Let $h\in\Hol(\Omega)$ satisfy
$h(0)=0$ and put $H=(R+\alpha)h$.  Then $H(0)=0$.  For fixed
$\xi\in\partial\Omega$,
\[
 \frac{\dd}{\dd r}\{r^\alpha h(r\xi)\}
 =
 r^{\alpha-1}H(r\xi).
\]
Hence
\[
 h(r\xi)
 =
 r^{-\alpha}\int_0^r t^{\alpha-1}H(t\xi)\,\dd t .
\]
We control the compact part using $H(0)=0$.  For $z\in(2/3)\Omega$,
\[
 h(z)=\int_0^1t^{\alpha-1}H(tz)\,\dd t.
\]
Choose $t_0>0$ as in the proof of
\cref{lem:radial-lp-companion}.  Its Cauchy estimate gives
\[
 \sup_{z\in(2/3)\Omega}|H(tz)|
 \le Ct\|H\|_{L^q((3/4)\Omega)},\qquad 0<t<t_0.
\]
For $t_0\le t\le1$, the change of variables $w=tz$ gives
\[
 \|H(t\,\cdot)\|_{L^q((2/3)\Omega)}
 \le t^{-2n/q}\|H\|_{L^q((3/4)\Omega)}.
\]
Minkowski's inequality now yields
\begin{equation}\label{eq:alpha-inner-estimate}
 \|h\|_{L^q((2/3)\Omega)}
 \le C_\alpha\|H\|_{L^q((3/4)\Omega)}.
\end{equation}
Both integrals in $t$ are finite, since the first has factor
$t^\alpha$ and the second starts at $t_0>0$.
The submean inequality near $\partial((1/2)\Omega)$ controls
the trace $h(\xi/2)$ by the right side of
\eqref{eq:alpha-inner-estimate}.

Put $r_0=1/2$.  On $r_0<r<1$ we write
\[
 h(r\xi)
 =
 r^{-\alpha}r_0^\alpha h(r_0\xi)
 +
 r^{-\alpha}\int_{r_0}^r t^{\alpha-1}H(t\xi)\,\dd t .
\]
The first term is controlled by the compact estimate.  In the second
term, both $r^{-\alpha}$ and $t^{\alpha-1}$ are bounded above and below
by constants depending only on $\alpha$ and $r_0$.  The same boundary
Hardy inequality as in \cref{lem:radial-lp-companion}, together with
$1-r\asymp\delta(r\xi)$, gives
\[
 \|h\|_{A^q}^q
 \le
 C_\alpha
 \int_\Omega |(R+\alpha)h(z)|^q\delta(z)^q\,\dd v(z).
\]
For the reverse estimate, Cauchy's inequality gives the same bound for
$\delta Rh$ as before, and
\[
 |(R+\alpha)h|^q\delta^q
 \le
 C_\alpha\{ |Rh|^q\delta^q+|h|^q\delta^q\}.
\]
The Whitney covering argument used in \cref{lem:radial-lp-companion}
therefore gives
\[
 \|h\|_{A^q}^q
 \asymp_\alpha
 \int_\Omega |(R+\alpha)h(z)|^q\delta(z)^q\,\dd v(z),
 \qquad 1<q<\infty,\quad h(0)=0 .
\]
For $h=C^{(\alpha)}_{g,\varphi}f$, the defining integral gives
$h(0)=0$, so the preceding estimate applies.  Also
\[
 C^{(\alpha)}_{g,\varphi}f(rz)
 =
 r^{-\alpha}
 \int_0^r Rf(\varphi(sz))g(sz)s^{\alpha-1}\,\dd s .
\]
Therefore
\[
 (R+\alpha)C^{(\alpha)}_{g,\varphi}f
 =
 g\,Rf\circ\varphi .
\]
It follows that
\[
\begin{aligned}
 \|C^{(\alpha)}_{g,\varphi}f\|_{A^q}^q
 &\asymp_\alpha
 \int_\Omega |Rf(\varphi(z))|^q|g(z)|^q\delta(z)^q\,\dd v(z)\\
 &=
 \int_\Omega |Rf(w)|^q\delta(w)^q\,\dd\mu_{g,\varphi,q}(w).
\end{aligned}
\]
The proofs of \cref{thm:p-le-q,thm:q-less-p} use only this reduction and
\cref{thm:derivative-carleson}.  They therefore apply verbatim, with
constants depending also on $\alpha$.  This proves the boundedness,
compactness, norm, and essential norm assertions.  When
$q=2$, the last equivalence says that
\[
 \|C^{(\alpha)}_{g,\varphi}f\|_{A^2}^2
 \asymp_\alpha
 \int_\Omega |Rf(w)|^2\delta(w)^2\,\dd\mu_{g,\varphi,2}(w).
\]
Thus $(C^{(\alpha)}_{g,\varphi})^*C^{(\alpha)}_{g,\varphi}$ is
comparable, in the quadratic-form sense, with the derivative Toeplitz
operator associated with $\mu_{g,\varphi,2}$.  Monotonicity of singular
values then gives the same Schatten criterion as in
\cref{thm:schatten-companion}, as well as the sufficient estimate in
\cref{cor:small-schatten-sufficient}.  The constants may depend on $\alpha$.
\end{proof}

\subsection{Compact-range symbols}

\begin{proposition}\label{cor:strict-self-map}
Let $\alpha\ge0$.  Let $g\in H^\infty(\Omega)$ and let
$\varphi\in\mathcal H(\Omega)$ satisfy $\varphi(0)=0$.  Assume that
there is a compact set $K\subset\Omega$ such that
\[
 \varphi(\Omega)\subset K .
\]
Then, for all $1<p,q<\infty$,
\[
 C^{(\alpha)}_{g,\varphi}:A^p(\Omega)\to A^q(\Omega)
\]
is compact.  Moreover, for every $0<s<\infty$,
\[
 C^{(\alpha)}_{g,\varphi}\in\Schatten_s(A^2(\Omega)).
\]
More precisely, there are constants $C,c>0$ such that
\[
 s_j(C^{(\alpha)}_{g,\varphi})
 \le C\exp(-c j^{1/n}),\qquad j\ge1.
\]
\end{proposition}

\begin{proof}
Fix $q$.  Since $K$ is compact in $\Omega$, there is $c_K>0$ such that
\[
 \delta(w)\ge c_K,\qquad w\in K .
\]
The measure $\mu_{g,\varphi,q}$ is carried by $K$ and has no atom at
$0$.  Moreover
\[
\begin{aligned}
 \mu_{g,\varphi,q}(\Omega)
 &=
 \int_{\{\varphi(z)\ne0\}}
 |g(z)|^q
 \left(\frac{\delta(z)}{\delta(\varphi(z))}\right)^q
 \dd v(z)\\
 &\le
 C_K\|g\|_\infty^q v(\Omega)<\infty .
\end{aligned}
\]
Thus $\mu_{g,\varphi,q}$ is a finite measure whose support is contained
in the compact set $K\subset\Omega$.  By the McNeal geometry, the set of
points $z$ for which
$\Q_\rho(z)$ meets this support is contained in another compact subset
of $\Omega$.  On this compact set $\V_\rho$ is bounded above and below.
Hence
\[
 \Dcal_{g,\varphi,p,q,\rho}(z)\to0,
 \qquad z\to\partial\Omega,
\]
when $p\le q$, and
\[
 \Dcal_{g,\varphi,p,q,\rho}\in L^r(\Omega,d\lambda_\Omega),
 \qquad r=\frac{pq}{p-q},
\]
when $q<p$.  The compactness follows from
\cref{thm:p-le-q,thm:q-less-p} when $\alpha=0$, and from
\cref{cor:alpha-family} when $\alpha>0$.

Choose $0<r_0<r_1<1$ with
\[
 K\Subset U_0:=r_0\Omega\Subset U_1:=r_1\Omega\Subset\Omega.
\]
Let $J:A^2(\Omega)\to A^2(U_0)$ be the restriction map.  Cauchy's
estimate on $U_1$ followed by Bernstein--Walsh polynomial approximation
on the nested convex sets $\overline{U_0}\Subset U_1$ gives the
following uniform bound: for each $f\in A^2(\Omega)$ and $N\ge0$,
there is a polynomial $p_N$ of total degree at most $N$ such that
\[
 \sup_{U_0}|f-p_N|\le C\theta^N\|f\|_{A^2(\Omega)},
 \qquad 0<\theta<1.
\]
Here $\overline{U_1}\Subset\Omega$, so Cauchy's estimate bounds
$\sup_{U_1}|f|$ by $C\|f\|_{A^2(\Omega)}$; the approximation
constants depend only on the fixed nested domains.  Consequently,
\begin{equation}\label{eq:compact-range-restriction}
 a_j(J)\le C\exp(-c j^{1/n}),\qquad j\ge1.
\end{equation}
Indeed, let $\mathcal P_N$ denote the polynomials of total degree at
most $N$, and let $P_N$ be the orthogonal projection of $A^2(U_0)$
onto $\mathcal P_N$.  Since $P_N$ gives the best approximation in
$A^2(U_0)$ from $\mathcal P_N$, the preceding bound gives
\[
 \|(I-P_N)Jf\|_{A^2(U_0)}
 \le C\theta^N\|f\|_{A^2(\Omega)}
\]
for some $0<\theta<1$.  Since
$\rank(P_NJ)\le\dim\mathcal P_N=\binom{N+n}{n}=O(N^n)$, choosing
$N\asymp j^{1/n}$ proves \eqref{eq:compact-range-restriction}.

The operator $C^{(\alpha)}_{g,\varphi}$ factors as $B_\alpha J$,
where $B_\alpha$ applies the defining integral to
$h\in A^2(U_0)$.  This is well defined since
$\varphi(\Omega)\subset K\Subset U_0$.  Cauchy's estimate gives
\[
 \sup_{w\in K}|Rh(w)|\le C_K\|h\|_{A^2(U_0)}.
\]
The radial Littlewood--Paley estimate, or its shifted version,
therefore gives
$\|B_\alpha h\|_{A^2(\Omega)}
\le C_{K,\alpha}\|g\|_\infty\|h\|_{A^2(U_0)}$.
The ideal property of approximation numbers and
\eqref{eq:compact-range-restriction} give
\[
 a_j(C^{(\alpha)}_{g,\varphi})
 \le \|B_\alpha\|a_j(J)
 \le C\exp(-c j^{1/n}).
\]
On Hilbert spaces, approximation numbers and singular values agree.
Hence
\[
 s_j(C^{(\alpha)}_{g,\varphi})
 \le C\exp(-c j^{1/n}).
\]
Summation proves membership in every $\Schatten_s$, $s>0$.
\end{proof}

\begin{corollary}\label{cor:no-schatten}
Let $C_g=C_{g,\operatorname{id}}$.  If $0<s<\infty$, then
\[
 C_g\in\Schatten_s(A^2(\Omega))
 \quad\Longleftrightarrow\quad
 g=0 .
\]
\end{corollary}

\begin{proof}
If $g=0$, the assertion is clear.  Conversely, assume
$C_g\in\Schatten_s$.  Then $C_g$ is compact.  The compactness part of
\cref{thm:p-le-q}, with $p=q=2$ and
$\varphi=\operatorname{id}$, gives
\[
 \Acal_{g,2,\rho}(z)\to0,
 \qquad z\to\partial\Omega.
\]
The sub-mean inequality and the maximum principle give $g=0$.
\end{proof}

\section*{Statements and Declarations}
\subsection*{Funding}
The first author was supported by the National Natural Science Foundation of
China (Grant No.~12601234).  The second author was supported by the National
Natural Science Foundation of China (Grant No.~12401154).

\subsection*{Data availability}
No datasets were generated or analyzed in this study.

\subsection*{Competing interests}
The authors declare that they have no competing interests.

\enlargethispage{3\baselineskip}
\par\smallskip
{\footnotesize
\noindent\textsc{Jianxiang Dong}, School of Mathematics and Statistics,
Tianshui Normal University, Tianshui 741000, China.
E-mail: \href{mailto:jianxd@tsnu.edu.cn}{jianxd@tsnu.edu.cn}.
\par\smallskip
\noindent\textsc{Chunxu Xu} (corresponding author), School of Science,
Nanjing Forestry University, Nanjing 210037, China.
E-mail: \href{mailto:1968385450@qq.com}{1968385450@qq.com}.
\par}

\end{document}